\documentclass[11pt,a4paper]{article}

\usepackage[english]{babel}
\usepackage[utf8]{inputenc} 
\usepackage[T1]{fontenc} 
\usepackage{amsmath,amsfonts,amssymb,amsthm}
\usepackage{textcomp}
\usepackage{stmaryrd}
\usepackage[a4paper]{geometry}
\usepackage{dsfont}
\usepackage{tikz}

\newcommand{\N}{\mathbb{N}}
\newcommand{\Z}{\mathbb{Z}}
\newcommand{\Q}{\mathbb{Q}}
\newcommand{\R}{\mathbb{R}}
\newcommand{\C}{\mathbb{C}}

\newcommand{\dps}{\displaystyle}
\newcommand{\pp}{\mathbb{P}}

\newcommand{\oo}{\mathcal{O}}

\newcommand{\gs}{\mathfrak{S}}

\newcommand{\Sc}{\mathbb{S}}
\newcommand{\Lb}{\mathcal{L}}
\newcommand{\Mb}{\mathcal{M}}
\newcommand{\Nb}{\mathcal{N}}
\newcommand{\fl}{\mathcal{F}\!\ell}

\DeclareMathOperator{\im}{Im}

\DeclareMathOperator{\rk}{rk}

\DeclareMathOperator{\spa}{Span}
\DeclareMathOperator{\gl}{GL}

\DeclareMathOperator{\h}{H}
\DeclareMathOperator{\conv}{conv}

\DeclareMathOperator{\pic}{Pic}
\DeclareMathOperator{\cl}{cl}

\DeclareMathOperator{\wt}{Wt}
\DeclareMathOperator{\proj}{Proj}
\DeclareMathOperator{\codim}{codim}
\DeclareMathOperator{\sym}{Sym}

\newtheorem{theo}{Theorem}[section]
\newtheorem{prop}[theo]{Proposition}
\newtheorem{lemma}[theo]{Lemma}

\theoremstyle{definition}
\newtheorem{de}[theo]{Definition}
\newtheorem{rmk}[theo]{Remark}

\newenvironment{changemargin}[2]{%
\begin{list}{}{%
\setlength{\topsep}{0pt}%
\setlength{\leftmargin}{#1}%
\setlength{\rightmargin}{#2}%
\setlength{\listparindent}{\parindent}%
\setlength{\itemindent}{\parindent}%
\setlength{\parsep}{\parskip}%
}%
\item[]}{\end{list}}

\title{A Geometric Approach to the stabilisation of certain sequences of Kronecker coefficients}
\author{Maxime Pelletier\thanks{Univ Lyon, Université Claude Bernard Lyon 1, CNRS UMR 5208, Institut Camille Jordan, 43 blvd. du 11 novembre 1918, F-69622 Villeurbanne cedex, France (\texttt{pelletier@math.univ-lyon1.fr})}}
\date{}

\begin{document}
\maketitle

\begin{abstract}
We give another proof, using tools from Geometric Invariant Theory, of a result due to S. Sam and A. Snowden in 2014, concerning the stability of Kronecker coefficients. This result states that some sequences of Kronecker coefficients eventually stabilise, and our method gives a nice geometric bound from which the stabilisation occurs. We perform the explicit computation of such a bound on two examples, one being the classical case of Murnaghan's stability. Moreover, we see that our techniques apply to other coefficients arising in Representation Theory: namely to some plethysm coefficients and in the case of the tensor product of representations of the hyperoctahedral group.
\end{abstract}

\section{Introduction}

For a positive integer $n$, let $\gs_n$ be the symmetric group over $n$ elements. The complex irreducible representations of this group are indexed by the partitions of $n$ (i.e. non-increasing finite sequences of positive integers -called parts- whose sum is equal to $n$). For a partition $\alpha$ of $n$ (for which the integer $n$ is called the size, and denoted $|\alpha|$), we denote its length (i.e. the number of parts) by $\ell(\alpha)$, and write $M_\alpha$ for the associated complex irreducible representation of $\gs_n$. An important problem concerning the representation theory of this group is the understanding of the decomposition of the tensor product of two such irreducible representations:
\[ M_\alpha\otimes M_\beta=\bigoplus_{\gamma\vdash n}M_\gamma^{\oplus g_{\alpha,\beta,\gamma}}, \]
where the multiplicities $g_{\alpha,\beta,\gamma}$ are non-negative integers, which are called the Kronecker coefficients. These coefficients appear in various situations, and are quite difficult to study. Some of their properties are nevertheless known, one of which being that the order of the three partitions indexing a Kronecker coefficient does not matter.

\vspace{5mm}

There are several different ways of studying the Kronecker coefficients, and we will be interested in their asymptotic behaviour, in various senses. They hold indeed a remarkable asymptotic property, noticed by F. Murnaghan in 1938: let $\alpha,\beta,\gamma$ be partitions of the same integer; if one repetitively increases by 1 the first part of each of these partitions, the corresponding sequence of Kronecker coefficients ends up stabilising. J. Stembridge, in \cite{stembridge}, introduced two notions of stability of a triple of partitions in order to generalise this Murnaghan's stability:

\begin{de}
A triple $(\alpha,\beta,\gamma)$ of partitions such that $|\alpha|=|\beta|=|\gamma|$ is called:
\begin{itemize}
\item weakly stable if $g_{d\alpha,d\beta,d\gamma}=1$ for all $d\in\N^*$;
\item stable if $g_{\alpha,\beta,\gamma}>0$ and, for any triple $(\lambda,\mu,\nu)$ of partitions such that $|\lambda|=|\mu|=|\nu|$, the sequence of general term $g_{\lambda+d\alpha,\mu+d\beta,\nu+d\gamma}$ is eventually constant.
\end{itemize}
\end{de}

The terminology ``weakly stable'' is in fact used by L. Manivel in \cite{manivel}. The notion of a stable triple is made to generalise the Murnaghan's stability: the latter simply means that the triple $\big((1),(1),(1)\big)$ is stable. By introducing the notion of a weakly stable triple, Stembridge hoped to find a more simple criterion to determine whether a triple is stable. He proved in \cite{stembridge} that a stable triple is weakly stable, and conjectured that the converse is true. S. Sam and A. Snowden proved shortly after, in \cite{sam-snowden}, that it is indeed verified.  We also learned during the redaction of this article about a prepublication by P.-E. Paradan \cite{paradan}, who demonstrated this kind of result in a more general context which in particular contains the case of Kronecker coefficients (as well as the plethysm case). In the first part of this article, we give another new proof of this result:

\begin{theo}\label{big_theo}
If a triple $(\alpha,\beta,\gamma)$ of partitions is weakly stable, then it is stable.
\end{theo}

A question then arises: given a stable triple, can we determine when the associated sequences of Kronecker coefficients do stabilise? There has already been results on this, at least in the case of Murnaghan's stability: for instance, M. Brion -in 1993- and E. Vallejo -in 1999- calculated bounds from which these sequences are necessarily constant. In \cite{briand-orellana-rosas}, E. Briand, R. Orellana, and M. Rosas recall the two bounds from Brion and Vallejo, and determine two other ones, still in the case of the stable triple $\big((1),(1),(1)\big)$.

\vspace{5mm}

The interesting aspect of our proof of Theorem \ref{big_theo} is that it gives a nice ``geometric bound'' from which we can be certain that the sequence $(g_{\lambda+d\alpha,\mu+d\beta,\nu+d\gamma})_d$ is constant, if the triple $(\alpha,\beta,\gamma)$ is stable. Indeed, the Kronecker coefficients can classically be related to the dimension of spaces of invariant sections from some line bundles: for all triples $(\alpha,\beta,\gamma)$ and $(\lambda,\mu,\nu)$, there exist a reductive group $G$ acting on a projective variety $X$, and two $G$-linearised line bundles $\Lb$ and $\Mb$ over $X$ whose spaces of invariant sections respectively give -via their dimension- the coefficients $g_{\alpha,\beta,\gamma}$ and $g_{\lambda,\mu,\nu}$ (cf. Section \ref{link_with_line_bundles}). Then, for $d\in\N$, $g_{\lambda+d\alpha,\mu+d\beta,\nu+d\gamma}$ is the dimension of $\h^0(X,\Mb+d\Lb)^G$, the space of invariant sections of the line bundle $\Mb+d\Lb$ on $X$.

\begin{prop}\label{prop_intro}
We suppose that the triple $(\alpha,\beta,\gamma)$ is weakly stable. Then:
\begin{itemize}
\item there exists an integer $D\in\N$ such that, for all $d\geq D$, $X^{ss}(\Mb+d\Lb)\subset X^{ss}(\Lb)$ where, if $\Nb$ is a line bundle, $X^{ss}(\Nb)$ stands for the set of semi-stable points with respect to $\Nb$, i.e. the points $x$ for which there exists an invariant section of a positive tensor power of $\Nb$ whose value at $x$ is not zero.
\item as soon as $X^{ss}(\Mb+d\Lb)\subset X^{ss}(\Lb)$, the Kronecker coefficient $g_{\lambda+d\alpha,\mu+d\beta,\nu+d\gamma}$ does not depend on $d$.
\end{itemize}
\end{prop}

What we prove precisely is in fact that $\h^0(X^{ss}(\Lb),\Mb+d\Lb)^G$ does not depend on $d$ and that, if $X^{ss}(\Mb+d\Lb)\subset X^{ss}(\Lb)$, then the restriction morphism $\h^0(X,\Mb+d\Lb)^G\hookrightarrow\h^0(X^{ss}(\Lb),\Mb+d\Lb)^G$ is an isomorphism. A natural question could thus be: is the converse true? We do not answer this question but, while the inclusion condition of semi-stable points applies to all $d$'s such that $\Mb+d\Lb$ belongs to a GIT-class containing the GIT-class of $\Lb$ in its closure, we manage in Section \ref{improvement} to extend the final result (i.e. the induced equality of dimensions between $\h^0(X,\Mb+d\Lb)^G$ and $\h^0(X^{ss}(\Lb),\Mb+d\Lb)^G$) to the $\Mb+d\Lb$'s in the boundary of these previous GIT-classes. The key points to obtain this extension are an argument of quasipolynomiality (which is a known result, of which we nevertheless write a proof in Section \ref{section_quasipolynomial}, inspired by \cite{kumar-prasad}) and the structure of the GIT-fan. The previous question could then raise another one: on these boundaries, is the restriction morphism $\h^0(X,\Mb+d\Lb)^G\hookrightarrow\h^0(X^{ss}(\Lb),\Mb+d\Lb)^G$ still an isomorphism?

\vspace{5mm}

In Section \ref{section_explicit_bounds}, we give a method allowing -at least for ``small'' weakly stable triples- to compute bounds from which the inclusion $X^{ss}(\Mb+d\Lb)\subset X^{ss}(\Lb)$ is realised. We perform the calculations for two examples of triples (namely $\big((1),(1),(1)\big)$ and $\big((1,1),(1,1),(2)\big)$). Taking into account the slight extension explained in the previous paragraph, it gives us:

\begin{theo}\label{thm_bound1}
If we denote $n_1=\ell(\lambda)$, $n_2=\ell(\mu)$, and set \footnote{The notation $\lceil x\rceil$ stands for the ceiling of the number $x$ (i.e. the integer such that $\lceil x\rceil -1<x\leq\lceil x\rceil$).}
\[ D_1=\left\lceil\frac{1}{2}\left(-\lambda_1+\lambda_2-\mu_1+\mu_2+2(\nu_2-\nu_{n_1n_2})+ \sum_{k=1}^{n_1+n_2-4}(\nu_{k+2}-\nu_{n_1n_2-k})\right)\right\rceil, \]
we have, for all $d\geq D_1$, $g_{\lambda+d(1),\mu+d(1),\nu+d(1)}=g_{\lambda+D_1(1),\mu+D_1(1),\nu+D_1(1)}$.
\end{theo}

\noindent (it is in this case legitimate to reorder the partitions $\lambda$, $\mu$, and $\nu$ to get the lowest bound $D_1$ possible) and

\begin{theo}\label{thm_bound2}
If $m=\max(-\lambda_2-\mu_1,-\lambda_1-\mu_2)$, and
\vspace{-5mm}
\begin{changemargin}{-2mm}{-2mm}
\[ D_2=\left\lbrace\begin{array}{ll}
\left\lceil\dps\frac{1}{2}\left(m+\lambda_3+ \mu_3+2(\nu_2-\nu_{n_1n_2})+\dps\sum_{k=1}^{n_1+n_2-4}(\nu_{k+2}-\nu_{n_1n_2-k})\right)\right\rceil & \text{if }n_1,n_2\geq 3\\
\left\lceil\dps\frac{1}{2}\left(m+\mu_3+2\nu_2-\nu_{2n_2}+\dps\sum_{k=1}^{n_2-1} \nu_{k+2}\right)\right\rceil & \text{if }n_1=2\\
\left\lceil\dps\frac{1}{2}\left(m+\lambda_3+2\nu_2-\nu_{2n_1}+\dps\sum_{k=1}^{n_1-1} \nu_{k+2}\right)\right\rceil & \text{if }n_2=2
\end{array}\right., \]
\end{changemargin}
then for all $d\geq D_2$, $g_{\lambda+d(1,1),\mu+d(1,1),\nu+d(2)}=g_{\lambda+D_2(1,1),\mu+D_2(1,1),\nu+D_2(2)}$.
\end{theo}

We then prove that our method allows to recover some of the bounds already existing in the case of Murnaghan's stability: we re-obtain Brion's bound, as well as the second one given by Briand, Orellana, and Rosas. Moreover, we get slight improvements for these in some cases. The bounds we obtained are in addition tested on some examples, in Section \ref{tests}. We also make a comparison on these examples with the four already existing bounds that we cited.

\vspace{5mm}

In Sections \ref{plethysm} and \ref{section_Bn}, using our method, we prove that weak stability also implies stability for some other coefficients arising in Representation Theory: at first for plethysm coefficients (the main result was already in \cite{sam-snowden} and \cite{paradan}), and then for the multiplicities in the tensor product of two irreducible representations of the hyperoctahedral group, which is the Weyl group of type $\mathrm{B}_n$.

\vspace{5mm}

\textit{Acknowledgements:} I would very much like to thank Nicolas Ressayre for extremely useful ideas and remarks during the preparation of this article. Thank you also to Michèle Vergne and Paul-Emile Paradan for pointing out interesting references and for their interest in this work.\\
This is a pre-print of an article published in \textit{manuscripta mathematica}. The final authenticated version is available online at: https://doi.org/10.1007/s00229-018-1021-4.\\
The author also acknowledges support from the French ANR (ANR project ANR-15-CE40-0012).

\section[Proof of Theorem 1.2 and Proposition 1.3]{Proof of Theorem \ref{big_theo} and Proposition \ref{prop_intro}}\label{general_section}

\subsection{Link with invariant sections of line bundles}\label{link_with_line_bundles}

Thanks to Schur-Weyl duality, the Kronecker coefficients also appear in the decomposition of representations of the general linear group. If $V_1$ and $V_2$ are two (complex) vector spaces, $\gamma$ is a partition, and if we denote by $\Sc$ the Schur functor\footnote{In other words, if $V$ is a complex vector space of dimension $n$, and $\lambda$ a partition of length $\leq n$, then $\Sc^\lambda(V)$ is the corresponding irreducible representation of $\gl(V)$. Moreover, all complex irreducible polynomial representations of this group are obtained this way.},
\[ \Sc^\gamma(V_1\otimes V_2)\simeq\bigoplus_{\alpha,\beta}\left(\Sc^\alpha(V_1)\otimes\Sc^\beta(V_2)\right)^{\oplus g_{\alpha,\beta,\gamma}} \]
as representations of $G=\gl(V_1)\times\gl(V_2)$. Then, by Schur's Lemma we have, for all triple $(\alpha,\beta,\gamma)$ of partitions (such that $|\alpha|=|\beta|=|\gamma|$) and all vector spaces $V_1$ and $V_2$ such that $\dim(V_1)\geq\ell(\alpha)$, $\dim(V_2)\geq\ell(\beta)$, and $\dim(V_1)\dim(V_2)\geq\ell(\gamma)$:
\[ g_{\alpha,\beta,\gamma}=\dim\left((\Sc^\alpha V_1)^*\otimes(\Sc^\beta V_2)^*\otimes\Sc^\gamma(V_1\otimes V_2)\right)^G. \]
Finally, we use Borel-Weil's Theorem: if $V$ is a complex vector space of finite dimension, we denote by $\fl(V)$ the complete flag variety associated to $V$. We know that, if $B$ is a Borel subgroup of $\gl(V)$, the variety $\fl(V)$ is isomorphic to $\gl(V)/B$. We can then define particular line bundles over $\gl(V)/B$: for any partition $\lambda$ of length at most $\dim V$, the finite sequence of integers $-\lambda$ defines a character $e^{-\lambda}$ of $B$, and this allows us to define $\Lb_\lambda=\gl(V)\times_B\C_{-\lambda}$, where $\C_{-\lambda}$ is the one-dimensional complex representation of $B$ given by the character $e^{-\lambda}$. The fibre product $\Lb_\lambda$ is a $\gl(V)$-linearised line bundle over $\gl(V)/B\simeq\fl(V)$. Then Borel-Weil's Theorem states that the representation $(\Sc^\alpha V_1)^*$ is isomorphic to $\h^0(\fl(V_1),\Lb_\alpha)$, the space of sections of the line bundle $\Lb_\alpha$ over $\fl(V_1)$. It is the same for $(\Sc^\beta V_2)^*$, and for $V_1\otimes V_2$ it yields $\Sc^\gamma(V_1\otimes V_2)\simeq\h^0(\fl(V_1\otimes V_2),\Lb_\gamma^*)$. Hence, we have the important following proposition:

\begin{prop}
For any triple $(\alpha,\beta,\gamma)$ of partitions such that $|\alpha|=|\beta|=|\gamma|$, there exist a reductive group G, a projective variety $X$ on which $G$ acts, and a $G$-linearised line bundle $\Lb_{\alpha,\beta,\gamma}$ over $X$ such that
\[ g_{\alpha,\beta,\gamma}=\dim\left(\h^0(X,\Lb_{\alpha,\beta,\gamma})^G\right). \]
\end{prop}

\begin{proof}
According to what precedes, it suffices to take $V_1$ and $V_2$ two vector spaces of large enough dimension, $G=\gl(V_1)\times\gl(V_2)$, $X=\fl(V_1)\times\fl(V_2)\times\fl(V_1\otimes V_2)$, and $\Lb_{\alpha,\beta,\gamma}=\Lb_\alpha\otimes\Lb_\beta\otimes\Lb_\gamma^*$.
\end{proof}

Thus, from now on, we consider a weakly stable triple $(\alpha,\beta,\gamma)$ of partitions, and another triple $(\lambda,\mu,\nu)$ of partitions (also satisfying $|\lambda|=|\mu|=|\nu|$). Then there exists a reductive group $G$, acting on a projective variety $X$, and two $G$-linearised line bundles $\Lb$ and $\Mb$ on $X$ such that:
\[ g_{\alpha,\beta,\gamma}=\dim\left(\h^0(X,\Lb)^G\right)\qquad\text{and}\qquad g_{\lambda,\mu,\nu}=\dim\left(\h^0(X,\Mb)^G\right) \]
(we denote by $V_1$ and $V_2$ the two vector spaces used to define those). We are interested in the behaviour of $\h^0(X,\Mb+d\Lb)^G$, or rather its dimension, for $d\in\N$.

\begin{rmk}
When we write $\Mb+d\Lb$, the operation denoted by ``$+$'' is the operation in $\pic^G(X)$, the group of $G$-linearised line bundles over $X$, i.e. the tensor product.
\end{rmk}

\subsection{Semi-stable points}\label{ss_points}

\subsubsection{Definition and criterion of semi-stability}

\begin{de}
Given a $G$-linearised line bundle $\Nb$ over $X$, we define the semi-stable points in $X$ (relatively to $\Nb$) as the elements of
\[ X^{ss}(\Nb)=\{x\in X\text{ s.t. }\exists k\in\N^*, \; \exists\sigma\in\h^0(X,\Nb^{\otimes k})^G,\; \sigma(x)\neq 0\}. \]
The points which are not semi-stable are said to be unstable (relatively to $\Nb$), and we denote by $X^{us}(\Nb)$ the set of unstable points.
\end{de}

Let us emphasise that this is not the standard definition of semi-stability (cf. for instance \cite{dolgachev}, Chapter 8): most often there is an additional requirement to fulfil for a point to be semi-stable. The definition we gave coincides nevertheless with the usual one in the case of an ample line bundle. The following result is then classical\footnote{Remark: it was proven by V. Guillemin and S. Sternberg in \cite{guillemin-sternberg}, before the article from Teleman which I cite thereafter.}:

\begin{prop}\label{prop_teleman}
If $\Nb$ is a $G$-linearised semi-ample line bundle over $X$, then
\[ \h^0(X,\Nb)^G\simeq\h^0(X^{ss}(\Nb),\Nb)^G. \]
\end{prop}

\begin{proof}
A proof of this result for ample line bundles can for example be found in \cite{teleman}, Theorem 2.11(a). C. Teleman gives it with the more usual definition of semi-stable points, which is not ours, but coincides with it in this case. Then, in the case of a semi-ample line bundle $\Nb$, there exists a $G$-equivariant projection $\pi:X\rightarrow\overline{X}$ (which is even a fibration with connected fibres) such that $\Nb$ is the pull-back by $\pi$ of an ample line bundle $\overline{\Nb}$ over a projective variety $\overline{X}$.\\
Indeed, $X$ is a product of flag varieties and, on such a variety, a semi-ample line bundle is a $\Lb_\delta$ for $\delta$ a partition. Moreover this $\Lb_\delta$ is ample if and only if the type of the partition (i.e. the indices $i$ such that $\delta_i>\delta_{i+1}$) coincides with the type of the flag variety. Henceforth, for every partition $\delta$, there exists a projection as announced above, which consists simply in forgetting in the flag variety the dimensions which do not appear in the type of $\delta$.

\vspace{5mm}

Then, with the properties of $\pi$,
\[ \h^0(X,\Nb)^G\simeq\h^0(\overline{X},\overline{\Nb})^G\simeq\h^0(\overline{X}^{ss}(\overline{\Nb}),\overline{\Nb})^G\simeq\h^0(X^{ss}(\Nb),\Nb)^G, \]
since $\pi^{-1}(\overline{X}^{ss}(\overline{\Nb}))=X^{ss}(\Nb)$.
\end{proof}

There is an extremely useful criterion of semi-stability which is called the Hilbert-Mumford criterion. It is generally stated for ample line bundles but, with the previously given definition of semi-stability, it holds for semi-ample line bundles (cf. \cite{ressayre}, Lemma~2), which is the case for all the line bundles we consider. We are going to rephrase this criterion to get a more geometric one, in terms of polytopes. Let us begin with the case in which a torus $T$ acts on $X$, and $\Nb$ is a $T$-linearised ample line bundle over $X$.

\vspace{5mm}

Then (see e.g. \cite{dolgachev}, Section 9.4), as $\Nb$ is ample, we have a closed embedding of $X$ in $\pp(V)$, where $V$ is a finite dimensional vector space, the action of $T$ on $X$ comes from a linear action on $V$, and some positive tensor power of $\Nb$ is the restriction of $\oo(1)$ to $X$. Then, since $T$ is a torus, $V$ splits into a direct sum of eigensubspaces,
\[ V=\bigoplus_{\chi\in X^*(T)} V_\chi, \]
where $X^*(T)$ denotes the set of all characters of $T$ and, for all $\chi\in X^*(T)$, $V_\chi=\{v\in V\text{ s.t. }\forall t\in T, \; t.v=\chi(t)v\}$ is the eigensubspace associated to the character $\chi$. Then, for $x\in X\subset\pp(V)$ and a $v=\sum_{\chi}v_\chi\in V$ ($v_\chi\in V_\chi$) such that $x=\spa(v)$, we define the weight set of $x$ as
\[ \wt(x)=\{\chi\in X^*(T)\text{ s.t. }v_\chi\neq 0\}. \]
Note that $\wt(x)$ is a finite subset of $X^*(T)\simeq\Z^N\subset\R^N$ ($N$ is the rank of $T$). We finally define the weight polytope of $x$ as the convex hull $\conv(\wt(x))$ of $\wt(x)$ in $\R^N$. Then, Theorem 9.2 of \cite{dolgachev} states that the Hilbert-Mumford criterion means:
\[ x\in X^{ss}(\Nb)\Longleftrightarrow 0\in\conv(\wt(x)). \]
We want to express this in a way which does not use an embedding in a $\pp(V)$, and which involves explicitly $\Nb$. For this, one has to wonder what any object in $\pp(V)$ corresponds to in $X$:
\[ \begin{array}{c|c}
\text{In }\pp(V) & \text{In }X\\
\hline
\pp(V_\chi)\text{ (for }V_\chi\neq\{0\}\text{)} & \text{fixed points of }T\\
\dps\bigcup_\chi(\pp(V_\chi)\cap X) & X^T=\{\text{fixed points of }T\text{ in }X\}\\
\pp(V_\chi)\cap X & \text{a union of some irreducible components $X_1,\dots,X_k$ of }X^T\\
\chi^{-1} & \text{character giving the action of $T$ on }\left.\Nb\right|_{X_i}\text{ for }i\in\llbracket 1,k\rrbracket
\end{array} \]
So we set, denoting by $X_1,\dots,X_s$ the irreducible components of $X^T$, for all $i\in\llbracket 1,s\rrbracket$,
\[ \begin{array}{rccl}
\chi_i: & \pic^T(X) & \longrightarrow & X^*(T)\\
 & \Nb & \longmapsto & \text{the inverse of the character giving the action of $T$ on }\left.\Nb\right|_{X_i}
\end{array}. \]
Then, the Hilbert-Mumford criterion states:
\[ x\in X^{ss}(\Nb)\Longleftrightarrow 0\in\conv(\{\chi_i(\Nb)\;;\;i\in\llbracket 1,s\rrbracket\text{ s.t. }\chi_i(\Nb)\text{ is a vertex of }\conv(\wt(x))\}). \]
And the only object left which uses an embedding of $X$ in $\pp(V)$ is $\wt(x)$. But we can get rid of it thanks to the following lemma:

\begin{lemma}
With the notations used above, if $x=\spa\left(\sum_{\chi}v_\chi\right)\in X\subset\pp(V)$,
\[ \chi\text{ is a vertex of }\conv(\wt(x))\Longleftrightarrow \pp(V_\chi)\cap\overline{T.x}\neq\emptyset. \]
\end{lemma}

\begin{proof}
Let us recall that there is a duality pairing between $X^*(T)$ and the one-parameter subgroups of $T$, whose set is denoted by $X_*(T)$: for all $\chi\in X^*(T)$ and $\tau\in X_*(T)$, $\chi\circ\tau:\C^*\rightarrow\C^*$ is of the form $z\mapsto z^n$ with $n$ integer. We set $\langle\chi,\tau\rangle=n$. Then, according to a classical property of convex polyhedra:
\vspace{-5mm}
\begin{changemargin}{-5mm}{-5mm}
\[ \chi\text{ is a vertex of }\conv(\wt(x))\Longleftrightarrow\exists\tau\in X_*(T)\text{ s.t. }\left\lbrace\begin{array}{l}
\langle\chi,\tau\rangle=0\\
\forall\chi'\in\conv(\wt(x))\setminus\{\chi\}, \; \langle\chi',\tau\rangle>0
\end{array}\right.. \]
\end{changemargin}
As a consequence, if $\chi$ is a vertex of $\conv(\wt(x))$, we have such a $\tau\in X_*(T)$. Moreover,
\[ \forall z\in\C^*, \: \tau(z).x=\spa\left(\sum_{\chi'}\chi'\circ\tau(z)v_{\chi'}\right)=\spa\left(\sum_{\chi'}z^{\langle\chi',\tau\rangle}v_{\chi'}\right). \]
And thus $\lim\limits_{z\to 0}(\tau(z).x)=\spa(v_{\chi})\in\pp(V_\chi)\cap\overline{T.x}$.

\vspace{5mm}

Conversely, if we suppose that $\chi$ is not a vertex of $\conv(\wt(x))$, then for all $\tau\in X_*(T)$, there exists $\chi(\tau)\in\conv(\wt(x))\setminus\{\chi\}$ such that $\langle\chi(\tau),\tau\rangle=\langle\chi,\tau\rangle$. We want to prove that $\pp(V_\chi)\cap\overline{T.x}=\emptyset$.\\
By contradiction, let us assume that $\pp(V_\chi)\cap\overline{T.x}\neq\emptyset$. Then there exists $\tau\in X_*(T)$ such that $\lim\limits_{z\to 0}(\tau(z).x)\in\pp(V_\chi)$. On the other hand,
\[ \forall z\in\C^*, \: \tau(z).x=\spa\left(\sum_{\chi'}z^{\langle\chi',\tau\rangle}v_{\chi'}\right). \]
So, for every $\chi'\in\wt(x)\setminus\{\chi\}$, $\langle\chi',\tau\rangle>\langle\chi,\tau\rangle$. This contradicts the existence of $\chi(\tau)$, which is necessarily a convex combination involving at least one element of $\wt(x)\setminus\{\chi\}$.
\end{proof}

Then,
\[ x\in X^{ss}(\Nb)\Longleftrightarrow 0\in\conv(\{\chi_i(\Nb)\;;\;i\in\llbracket 1,s\rrbracket\text{ s.t. }X_i\cap\overline{T.x}\neq\emptyset\}), \]
which now does not involve anymore any embedding of $X$ in $\pp(V)$. So this is also true for line bundles which are semi-ample, and not necessarily ample (since Hilbert-Mumford criterion holds for such ones). We now extend this to the case when $G$ is reductive. Then we take a maximal torus $T$ in $G$ and, using Theorem 9.3 of \cite{dolgachev}, we finally get:
\begin{prop}
In our settings (a reductive group $G$ acting on a flag variety $X$), if $\Nb$ is a $G$-linearised semi-ample line bundle over $X$, then
\[ x\in X^{ss}(\Nb)\Longleftrightarrow\forall g\in G, \: 0\in\conv(\{\chi_i(\Nb)\;;\;i\in\llbracket 1,s\rrbracket\text{ s.t. }X_i\cap\overline{T.(g.x)}\neq\emptyset\}), \]
where $T$ is a maximal torus in $G$, and $X_1,\dots,X_s$ are the irreducible components of $X^T$.
\end{prop}

\subsubsection{Inclusions of sets of semi-stable points}

The following proposition could be deduced from well-known results on the GIT-fan (see e.g. \cite{dolgachev-hu}, Section 3.4, or \cite{ressayre2}, Section 5), but we give another proof specific to this case:

\begin{prop}\label{prop_semi-stable}
There exists $D\in\N$ such that, for all $d\geq D$, $X^{ss}(\Mb+d\Lb)\subset X^{ss}(\Lb)$.
\end{prop}

\begin{proof}
To all $x\in X$ and $g\in G$, we associate $E_{x,g}\in\mathcal{P}(\llbracket 1,s\rrbracket)$ (i.e. a subset of $\llbracket 1,s\rrbracket$) as follows:
\[ E_{x,g}=\{i\in\llbracket 1,s\rrbracket\text{ s.t. }X_i\cap\overline{T.(g.x)}\neq\emptyset\}. \]
With this notation, we know that:
\[ x\in X^{ss}(\Lb)\Longleftrightarrow\forall g\in G, \: 0\in\conv(\{\chi_i(\Lb)\;;\;i\in E_{x,g}\}). \]
So we set $A=\left\lbrace E_{x,g}\text{ s.t. }0\notin\conv(\{\chi_i(\Lb)\;;\;i\in E_{x,g}\})\right\rbrace$, which is finite since contained in $\mathcal{P}(\llbracket 1,s\rrbracket)$. Then, for all $E\in A$, there exists $\varphi_E\in(\R^N)^*$ such that, for all $i\in E$, $\varphi_E(\chi_i(\Lb))>0$ (by Hahn-Banach Theorem). Moreover\footnote{the applications $\varphi_E\circ\chi_i:\pic^G(X)\rightarrow\R$ can be extended without problem to $\pic^G(X)\otimes_\Z\Q$},
\[ \forall E\in A, \; \forall i\in E, \: \varphi_E\circ\chi_i\left(\frac{\Mb+d\Lb}{d}\right)\xrightarrow[d\to +\infty]{}\varphi_E(\chi_i(\Lb))>0, \]
so there exists $D_E\in\N^*$ such that, for all $d\geq D_E$, for all $i\in E$, $\varphi_E\circ\chi_i\left(\dps\frac{\Mb+d\Lb}{d}\right)>0$.

\vspace{5mm}

We then set $D=\max\{D_E\;;\;E\in A\}$. Let $d\in\N$, $d\geq D$. Let $x\notin X^{ss}(\Lb)$, which means that there exists $g\in G$ such that $0\notin\conv(\{\chi_i(\Lb)\;;\;i\in E_{x,g}\})$. In other words, $E_{x,g}\in A$. So, as $d\geq D\geq D_{E_{x,g}}$, $\varphi_{E_{x,g}}(\chi_i(\Mb+d\Lb))=d\varphi_{E_{x,g}}\circ\chi_i\left(\dps\frac{\Mb+d\Lb}{d}\right)>0$ for all $i\in E_{x,g}$. Hence (once again by Hahn-Banach Theorem),
\[ 0\notin\conv\left(\left\lbrace\chi_i\left(\Mb+d\Lb\right)\;;\;i\in E_{x,g}\right\rbrace\right), \text{ i.e. }x\notin X^{ss}(\Mb+d\Lb). \]
Thus, $X^{ss}(\Mb+d\Lb)\subset X^{ss}(\Lb)$.
\end{proof}

\subsection[Use of Luna's Slice Etale Theorem]{Use of Luna's Slice \'{E}tale Theorem}

Let us recall that we considered a triple of partitions $(\alpha,\beta,\gamma)$ such that, for all $d\in\N^*$, $g_{d\alpha,d\beta,d\gamma}=1$. This means that,
\[ \forall d\in\N^*, \: \h^0(X,\Lb^{\otimes d})^G\simeq\C. \]
Then, using Proposition 8.1 of \cite{dolgachev}, as $X$ is projective,
\[ X^{ss}(\Lb)\sslash G\simeq\proj(\C[t]). \]
So $X^{ss}(\Lb)\sslash G$ is a point. Thus $X^{ss}(\Lb)$ contains exactly one closed $G$-orbit, denoted by $G.x_0$. Moreover, $X^{ss}(\Lb)$ is affine (since the canonical projection $X^{ss}(\Lb)\rightarrow X^{ss}(\Lb)\sslash G$ is affine). So we can use Corollary 2 to Luna's Slice \'{E}tale Theorem (cf. \cite{luna}): there exist a reductive subgroup $H$ -which is in fact the isotropy subgroup $G_{x_0}$- of $G$ and an affine $H$-variety $S$ such that
\[ \left\lbrace\begin{array}{l}
S^H=\{x_0\}\\
\forall x\in S, \; x_0\in\overline{H.x}\\
X^{ss}(\Lb)\simeq G\times_H S
\end{array}\right.. \]
Furthermore, $S$ is a (complex) vector space of finite dimension on which $H$ acts linearly.

\subsection[Proof of Theorem 1.2]{Proof of Theorem \ref{big_theo}}\label{proof_big_theo}

We are now ready to prove Theorem \ref{big_theo}. We still have our weakly stable triple $(\alpha,\beta,\gamma)$ and another triple of partitions $(\lambda,\mu,\nu)$, which give rise to the two (semi-ample) line bundles $\Lb$ and $\Mb$.

\begin{prop}\label{inclusion_of_ss_points_suffices}
If $D\in\N$ is such that, for all $d\geq D$, $X^{ss}(\Mb+d\Lb)\subset X^{ss}(\Lb)$, then
\[ \forall d\geq D, \: \h^0(X,\Mb+d\Lb)^G\simeq\h^0(S,\Mb)^H. \]
\end{prop}

\begin{proof}
Let $D\in\N$ be as in the statement, and $d\in\N$, $d\geq D$. Then, thanks to Proposition \ref{prop_teleman},
\[ \h^0(X,\Mb+d\Lb)^G\simeq\h^0(X^{ss}(\Mb+d\Lb),\Mb+d\Lb)^G. \]
Consequently, since $X^{ss}(\Mb+d\Lb)\subset X^{ss}(\Lb)\subset X$,
\[ \h^0(X,\Mb+d\Lb)^G\simeq\h^0(X^{ss}(\Lb),\Mb+d\Lb)^G. \]
Now, using the consequence of Luna's Slice \'{E}tale Theorem:
\[ \h^0(X,\Mb+d\Lb)^G\simeq\h^0(G\times_H S,\Mb+d\Lb)^G\simeq\h^0(S,\Mb+d\Lb)^H. \]
We are almost done; it only remains to prove that $\h^0(S,\Mb+d\Lb)^H$ does not depend on $d$. For this, we demonstrate that $\Lb$ is trivial on $S$, using the following lemma:

\begin{lemma}
The application
\[ \begin{array}{rccl}
\psi: & X^*(H) & \longrightarrow & \pic^H(S)\\
 & \chi & \longmapsto & \Lb_\chi
\end{array}, \]
where $\Lb_\chi$ is the trivial bundle $S\times\C$ whose $H$-linearisation is given by the character $\chi$, is an isomorphism.
\end{lemma}

\begin{proof}
The only non trivial thing to prove is the surjectivity of $\psi$. Let $\Nb\in\pic^H(S)$. We have seen that $x_0$ is a point of $S$ fixed by $H$. So, $H$ acts on the fibre $\Nb_{x_0}$. This action gives $\chi\in X^*(H)$. Moreover, $\Nb$ is trivial because $S$ is a vector space. Necessarily, its linearisation is given by the character $\chi$.
\end{proof}

We consider the character $\chi_0$ given by the action of $H$ on $\Lb_{x_0}$ and we want to prove that $\chi_0$ is trivial. As $x_0\in X^{ss}(\Lb)$, there exist $k\in\N^*$ and $\sigma\in\h^0(X,\Lb^{\otimes k})^G$ such that $\sigma(x_0)\neq 0$. Moreover, $\dim(\h^0(X,\Lb)^G)=\dim(\h^0(X,\Lb^{\otimes k})^G)=1$ so, if we take $\sigma_0\in\h^0(X,\Lb)^G\setminus\{0\}$, we have $\sigma_0^{\otimes k}=t\sigma$ with $t\in\C^*$. As a consequence, $\sigma_0^{\otimes k}(x_0)\neq 0$ and so $\sigma_0(x_0)\neq 0$.\\
Furthermore,
\[ \forall h\in H, \: \sigma_0(x_0)=\sigma_0(h.x_0)=h.\sigma_0(x_0)=\chi_0(h)\sigma_0(x_0), \]
and then $\chi_0(h)=1$ for all $h\in H$. Thus, $\chi_0$ is trivial and so is $\Lb	$ over $S$.

\vspace{5mm}

Finally,
\[ \forall d\geq D, \: \h^0(X,\Mb+d\Lb)^G\simeq\h^0(S,\Mb)^H, \]
\end{proof}

\noindent Proposition \ref{prop_semi-stable} and Proposition \ref{inclusion_of_ss_points_suffices} together conclude the proof of Proposition \ref{prop_intro}, and as a consequence our proof of Theorem \ref{big_theo}. Let us note that the formula we get for the limit coefficient (i.e. $\dim\h^0(S,\Mb)^H$) corresponds to the one obtained by Paradan in \cite{paradan}, Theorem 5.12.

\section{Explicit bounds in some specific cases}\label{section_explicit_bounds}

We saw in the previous section that the sequence $(g_{\lambda+d\alpha,\mu+d\beta,\nu+d\gamma})_{d\in\N}$ stabilises as soon as $X^{ss}(\Mb+d\Lb)\subset X^{ss}(\Lb)$. We now would like to see if one can compute the rank $D$ from which this inclusion is realised. The computation of the $D$ from Proposition \ref{prop_semi-stable} appears to be too tricky, and so in the following we focus on two examples in which we can do explicit computations using another method.

\subsection{Steps of the computation}\label{steps_computation}

The inclusion $X^{ss}(\Mb+d\Lb)\subset X^{ss}(\Lb)$ we are interested in is equivalent to the following: $X^{us}(\Lb)\subset X^{us}(\Mb+d\Lb)$. Here we are rather looking to prove this last one, principally because we find that the fact of being an unstable point has -thanks to the Hilbert-Mumford criterion- a more practical description. Here are the different steps we are then going to carry out on the two examples:
\begin{itemize}
\item The first step is to consider the projection $\pi:X\rightarrow\overline{X}$ onto the product of partial flag varieties such that $\Lb$ is the pull-back of an ample line bundle $\overline{\Lb}$ over $\overline{X}$.
\item The second step is to study the set $\overline{X}^{us}(\overline{\Lb})$ of unstable points in $\overline{X}$. More precisely, we want to express this set as the union of some orbit closures: $\cl(G.\overline{x_1}),\dots,$ $\cl(G.\overline{x_p})$.
\item Then one can prove that, thanks to good properties of the projection $\pi$, $X^{us}(\Lb)$ is the union of the closures of $\pi^{-1}(G.\overline{x_1}),\dots,\pi^{-1}(G.\overline{x_p})$. As a consequence, since $X^{us}(\Mb+d\Lb)$ is closed and $\pi$ is $G$-equivariant, to prove that $X^{us}(\Lb)\subset X^{us}(\Mb+d\Lb)$ we only need to show for all $i\in\llbracket 1,p\rrbracket$ that $\pi^{-1}(\overline{x_i})\subset X^{us}(\Mb+d\Lb)$.
\item In the fourth step we want to use the Hilbert-Mumford criterion. Let us write it in a way different from before:

\begin{de}
Let $Y$ be a projective variety on which a reductive group $H$ acts, and $\Nb$ a $H$-linearised line bundle over $Y$. Let $y\in Y$ and $\tau$ be a one-parameter subgroup of $H$ (denoted $\tau\in X_*(H)$). Since $Y$ is projective, $\lim\limits_{t\to 0}\tau(t).y$ exists. We denote it by $z$. This point is fixed by the image of $\tau$, and so $\C^*$ acts via $\tau$ on the fibre $\Nb_z$. Then there exists an integer $\mu^\Nb(y,\tau)$ such that, for all $t\in\C^*$ and $\tilde{z}\in\Nb_z$,
\[ \tau(t).\tilde{z}=t^{-\mu^\Nb(y,\tau)}\tilde{z}. \]
\end{de}

The Hilbert-Mumford criterion can then be stated as (see e.g. \cite{ressayre}, Lemma 2):

\begin{prop}
In the settings of the previous definition, if in addition $\Nb$ is semi-ample, then:
\[ y\in Y^{ss}(\Nb)\quad\Longleftrightarrow\quad\forall\tau\in X_*(H), \; \mu^\Nb(y,\tau)\leq 0. \]
\end{prop}

Set $i\in\llbracket 1,p\rrbracket$. Since $\overline{x_i}\in\overline{X}^{us}(\overline{\Lb})$, we can find a destabilising one-parameter subgroup for $\overline{x_i}$: $\tau_i$ such that $\mu^{\overline{\Lb}}(\overline{x_i},\tau_i)>0$.
\item Let us keep in mind that we want to get $\pi^{-1}(\overline{x_i})\subset X^{us}(\Mb+d\Lb)$. By Hilbert-Mumford criterion, this will be true when, for all $x\in\pi^{-1}(\overline{x_i})$, $\mu^{\Mb+d\Lb}(x,\tau_i)>0$. But, for such an $x$, we have:
\[ \mu^{\Mb+d\Lb}(x,\tau_i)=\mu^{\Mb}(x,\tau_i)+d\mu^{\overline{\Lb}}(\overline{x_i},\tau_i). \]
So we only need to calculate $\mu^{\Mb}(x,\tau_i)$ for all $x\in\pi^{-1}(\overline{x_i})$:
\begin{itemize}
\item From the definition of the integers $\mu^{\Mb}(.,\tau_i)$, we see that we can restrict to the case when $x\in\pi^{-1}(\overline{x_i})$ is a fixed point of $\tau_i$. Then at first we determine the form of such a fixed point.
\item Finally we calculate explicitly the action of $\tau_i$ on the fibre of $\Mb$ over such a point.
\end{itemize}
\item As a conclusion, as soon as
\[ d>-\frac{\mu^{\Mb}(x,\tau_i)}{\mu^{\overline{\Lb}}(\overline{x_i},\tau_i)} \]
for all $i\in\llbracket 1,p\rrbracket$ and $x\in\pi^{-1}(\overline{x_i})^{\tau_i}$, we have the inclusion we were looking for.
\end{itemize}

\subsection{Case of Murnaghan's stability}\label{example_murnaghan}

\subsubsection{Reduction to ample line bundles}

In this case, the stable triple we are interested in is simply $\big((1),(1),(1)\big)$. It has been known for a long time that it is a stable triple. Consider
\vspace{-5mm}
\begin{changemargin}{-2mm}{-2mm}
\[ \begin{array}{rccl}
\pi: & X & \longrightarrow & \dps\overset{\text{denoted }\overline{X}}{\overbrace{\pp(V_1)\times\pp(V_2)\times\pp((V_1\otimes V_2)^*)}}\\
 & \left((W_{1,i})_i,(W_{2,i})_i,(W'_i)_i\right) & \longmapsto & (W_{1,1},W_{2,1},\{\varphi\in(V_1\otimes V_2)^*\text{ s.t. }\ker\varphi=W'_{n_1n_2-1}\})
\end{array}. \]
\end{changemargin}
Since $\alpha=\beta=\gamma=(1)$, we have that $\Lb=\Lb_\alpha\otimes\Lb_\beta\otimes\Lb^*_\gamma$ is the pull-back of $\oo(1)\otimes\oo(1)\otimes\oo(1)$ (denoted $\overline{\Lb}$ from now on) by $\pi$. Moreover,
\[ \mathrm{H}^0(\overline{X},\overline{\Lb})^G\simeq (V_1^*\otimes V_2^*\otimes V_1\otimes V_2)^G\simeq\C. \]
So $\overline{X}^{ss}(\overline{\Lb})=\{x\in\overline{X}\text{ s.t. }\overline{\sigma}_0(x)\neq 0\}$ for any $\overline{\sigma}_0\in\mathrm{H}^0(\overline{X},\overline{\Lb})^G\setminus\{0\}$. A simple non-zero section on $\overline{X}$ is
\[ \C v_1\otimes\C v_2\otimes\C\varphi\longmapsto\varphi(v_1\otimes v_2). \]
And
\[ \overline{X}^{ss}(\overline{\Lb})=\{(\C v_1,\C v_2,\C\varphi)\in\overline{X}\text{ s.t. }v_1\otimes v_2\notin\ker\varphi\}. \]

\subsubsection{Determination of $\overline{X}^{us}(\overline{\Lb})$}\label{Xus_for_Murnaghan}

Let us take $(e_1,\dots,e_{n_1})$ a basis in $V_1$ (with $n_1\geq 2$), and $(f_1,\dots,f_{n_2})$ a basis in $V_2$ ($n_2\geq 2$). Their dual bases are denoted with upper stars. Moreover, we set $n=\min(n_1,n_2)$.

\begin{prop}\label{orbits_murnaghan}
The set $\overline{X}^{us}(\overline{\Lb})$ consists in the closure of the orbit of the element $\overline{x}=(\C e_1,\C f_2,\C \varphi_n)$, where $\varphi_n=\sum_{i=1}^n e_i^*\otimes f_i^*\in V_1^*\otimes V_2^*\simeq(V_1\otimes V_2)^*$.
\end{prop}

\begin{proof}At first, since $\overline{X}^{us}(\overline{\Lb})=\{(\C v_1,\C v_2,\C\varphi)\in\overline{X}\text{ s.t. }\varphi(v_1\otimes v_2)=0\}$, $\overline{X}^{us}(\overline{\Lb})$ is pure of codimension 1.

\vspace{5mm}

Then $\pp((V_1\otimes V_2)^*)\simeq\pp(V_1^*\otimes V_2^*)\simeq\pp(\mathrm{Hom}(V_1,V_2^*))$. So we consider $(l_1,l_2,\C\psi)\in\pp(V_1)\times\pp(V_2)\times\pp(\mathrm{Hom}(V_1,V_2^*))$. The action of $G$ is then:
\[ \forall(g_1,g_2)\in G, \; (g_1,g_2).(l_1,l_2,\C\psi)=(g_1(l_1),g_2(l_2),\C {}^t\!g_2^{-1}\circ\psi\circ g_1^{-1}). \]
So we know that the orbits of the action on the third part ($\C\psi$) are classified by the rank of $\psi$. Moreover, this triple $(l_1,l_2,\C\psi)$ defines several subspaces:
\[ \begin{array}{c|c|c|c}
\text{in }V_1 & \text{in }V_2 & \text{in }V_1^* & \text{in }V_2^*\\
\hline
l_1 & l_2 & H_1=l_1^\perp & H_2=l_2^\perp\\
\ker\psi & \ker{}^t\psi & \im{}^t\psi=(\ker\psi)^\perp & \im\psi=(\ker{}^t\psi)\perp\\
\psi^{-1}(H_2) & {}^t\psi^{-1}(H_1) & {}^t\psi(l_2)=\psi^{-1}(H_2)^\perp & \psi(l_1)={}^t\psi^{-1}(H_1)^\perp
\end{array} \]
and the different possible positions of $l_1$ and $l_2$ with respect to $\ker\psi$, $\psi^{-1}(H_2)$, and respectively $\ker{}^t\psi$, ${}^t\psi^{-1}(H_1)$, shall help us to describe the orbits. Furthermore, $\ker\psi\subset\psi^{-1}(H_2)$, $\ker{}^t\psi\subset{}^t\psi^{-1}(H_1)$, and $l_1\subset\psi^{-1}(H_2)\Leftrightarrow l_2\subset{}^t\psi^{-1}(H_1)$.

\vspace{5mm}

\noindent\underline{First case:} $n_1=n_2$ (so $n=n_1=n_2$).\\
Let us first assume that $\rk\psi=n$. Then, $\ker\psi=\{0\}$ and $\ker{}^t\psi=\{0\}$. So this leaves two possibilities for the positions of $l_1$ and $l_2$:
\begin{itemize}
\item $l_1\subset\psi^{-1}(H_2)$ and $l_2\subset{}^t\psi^{-1}(H_1)$. One can check that such $(l_1,l_2,\C\psi)$ form one orbit, $\oo_1$.
\item $l_1\not\subset\psi^{-1}(H_2)$ and $l_2\not\subset{}^t\psi^{-1}(H_1)$. One can also check that such triples form a second orbit, $\oo_2$.
\end{itemize}
We can see that $\oo_1$ is unstable, whereas $\oo_2$ is semi-stable.\\
What if $\rk\psi\leq n-1$? The closed subset $Y=\{(l_1,l_2,\C\psi)\text{ s.t. }\rk\psi\leq n-1\}$ satisfies $\codim(Y\cap \overline{X}^{us}(\overline{\Lb}))\geq 2$ because, for all $l_1$ and $l_2$, $\{\C\psi\,;\,\rk\psi\leq n-1\text{ and }\psi(l_1)(l_2)=\{0\}\}$ has codimension 2 in $\pp(\mathrm{Hom}(V_1,V_2^*))$. So the complementary of $Y\cap\overline{X}^{us}(\overline{\Lb})$ intersects every irreducible components of $\overline{X}^{us}(\overline{\Lb})$. Thus, $Y^c=\{(l_1,l_2,\C\psi)\text{ s.t. }\rk\psi=n\}$ intersects every irreducible components of $\overline{X}^{us}(\overline{\Lb})$\\
\textit{Conclusion for this case:} $\overline{X}^{us}(\overline{\Lb})=\cl(\oo_1)$, the closure of orbit $\oo_1$. Furthermore, a representative of $\oo_1$ is $\overline{x}=(\C e_1,\C f_2,\C\varphi_n)$.

\vspace{5mm}

\noindent\underline{Second case:} $n_1<n_2$ (and then $n=n_1$).\\
In this case, $\{\C\psi\text{ s.t. }\rk\psi\leq n-1\}$ has codimension at least 2 (because the minors of rank $n$ must be zero, and there are at least 2). So, as in the previous case, it suffices to consider the case where $\rk\psi=n$, for which $\ker\psi=\{0\}$ and $\ker{}^t\psi\neq\{0\}$. This leads to three possibilities for $l_1$ and $l_2$:
\begin{itemize}
\item $l_1\subset\psi^{-1}(H_2)$ and $l_2\subset\ker{}^t\psi \subset{}^t\psi^{-1}(H_1)$. One can check that such $(l_1,l_2,\C\psi)$ form one orbit, $\oo_1$.
\item $l_1\subset\psi^{-1}(H_2)$ and $l_2\not\subset\ker{}^t\psi$, but $l_2\subset{}^t\psi^{-1}(H_1)$. Once again, one can check that this gives only one orbit, $\oo_2$.
\item $l_1\not\subset\psi^{-1}(H_2)$ and $l_2\not\subset{}^t\psi^{-1}(H_1)$. One can still check that these triples form one orbit, $\oo_3$.
\end{itemize}
The orbit $\oo_3$ is semi-stable, whereas $\oo_1$ and $\oo_2$ are unstable. In addition, $\oo_1\subset\cl(\oo_2)$ because, if $\rk\psi=n$ and $(l_1,l_2,\C\psi)$ is unstable, $(l_1,l_2,\C\psi)\in\oo_1\Leftrightarrow{}^t\psi(l_2)=\{0\}$ and $(l_1,l_2,\C\psi)\in\oo_2\Leftrightarrow{}^t\psi(l_2)\neq\{0\}$.\\
\textit{Conclusion for that case:} Here, $\overline{X}^{us}(\overline{\Lb})=\cl(\oo_2)$ and a representative of $\oo_2$ is the same $\overline{x}$ as before: $\overline{x}=(\C e_1,\C f_2,\C\varphi_n)$.

\vspace{5mm}

\noindent\underline{Third and last case:} $n_1>n_2$.\\
Everything happens similarly to the previous case, if we exchange the roles of $V_1$ and $V_2$. So we have also the orbit of $\overline{x}=(\C e_1,\C f_2,\C\varphi_n)$ which is dense in $\overline{X}^{us}(\overline{\Lb})$.
\end{proof}

\subsubsection[Restriction to π-1(x)]{Restriction to $\pi^{-1}(\overline{x})$}\label{restriction_pi-1(x)}

The projection $\pi$ we use is of the form
\[ \pi:\tilde{G}/\tilde{B}\longrightarrow\tilde{G}/\tilde{P}, \]
with $\tilde{G}$ a complex reductive group, $\tilde{B}$ a Borel subgroup, and $\tilde{P}$ a parabolic subgroup containing $\tilde{B}$. So the fibres are all isomorphic to $\tilde{P}/\tilde{B}$ ($\pi$ is even a fibration). This is also true for its restriction to $X^{us}(\Lb)=\pi^{-1}(\overline{X}^{us}(\overline{\Lb}))$. Thus, since $G.\overline{x}$ is dense in $\overline{X}^{us}(\overline{\Lb})$, $\pi^{-1}(G.\overline{x})$ is dense in $X^{us}(\Lb)$. As a consequence, $X^{us}(\Lb)\subset X^{us}(\Mb+d\Lb)$ if $\pi^{-1}(G.\overline{x})\subset X^{us}(\Mb+d\Lb)$ (because $X^{us}(\Mb+d\Lb)$ is closed). And finally, if $\pi^{-1}(\overline{x})\subset X^{us}(\Mb+d\Lb)$, then $\pi^{-1}(G.\overline{x})\subset X^{us}(\Mb+d\Lb)$ since $\pi$ is $G$-equivariant. Hence the following lemma:

\begin{lemma}\label{x_bar_suffices}
If $d_0\in\N$ is such that, for all $d\geq d_0$, $\pi^{-1}(\overline{x})\subset X^{us}(\Mb+d\Lb)$, then
\[ \forall d\geq d_0, \: g_{\lambda+d(1),\mu+d(1),\nu+d(1)}=g_{\lambda+d_0(1),\mu+d_0(1),\nu+d_0(1)}. \]
\end{lemma}

\subsubsection{Computation of the bound}\label{computation_murnaghan}

We identify $\gl(V_1)$, $\gl(V_2)$, and $\gl(V_1\otimes V_2)$ respectively with $\gl_{n_1}(\C)$, $\gl_{n_2}(\C)$, and $\gl_{n_1n_2}(\C)$ thanks to the bases given in section \ref{Xus_for_Murnaghan}. The basis in $V_1\otimes V_2$ is then $(e_i\otimes f_j)_{i,j}$, ordered lexicographically. Moreover we use the following notation for one-parameter subgroups of some $\gl_{m_1}(\C)\times\dots\times\gl_{m_p}(\C)$:
\[ \begin{array}{rccl}
\tau: & \C^* & \longrightarrow & \gl_{m_1}(\C)\times\dots\times\gl_{m_p}(\C)\\
 & t & \longmapsto & (\begin{pmatrix}
t^{a^{(1)}_1} &&&\\
 & t^{a^{(1)}_2} && \\
 && \ddots & \\
 &&& t^{a^{(1)}_{m_1}}
\end{pmatrix},\dots,\begin{pmatrix}
t^{a^{(p)}_1} &&&\\
 & t^{a^{(p)}_2} && \\
 && \ddots & \\
 &&& t^{a^{(p)}_{m_p}}
\end{pmatrix})
\end{array} \]
is denoted by $\tau = \left( a^{(1)}_1 , a^{(1)}_2 , \dots , a^{(1)}_{m_1} | \dots | a^{(p)}_1 , \dots , a^{(p)}_{m_p} \right)$.

\vspace{5mm}

\noindent\underline{Destabilising one-parameter subgroup for $\overline{x}$:} We set the following one-parameter subgroup of $G$:
\[ \tau_0=\left.\big(1,-1,0,\dots,0\text{ }\right|-1,1,0,\dots,0\big). \]
Then, since the action of $\tau_0(t)$ on the lines $\C e_1$, $\C f_2$, and $\C\varphi_n$ is the multiplication by $t$, $t$, and $1$ respectively, we have
\[ \mu^{\overline{\Lb}}(\overline{x},\tau_0)=2. \]
Let now $x\in\pi^{-1}(\overline{x})$. We want to calculate $\mu^{\Mb}(x,\tau_0)$. Thanks to the way $\mu$ is defined (first, one has to take the limit when $t\to 0$ from $\tau_0(t).x$ and gets a fixed point of $\tau_0$), and since $\overline{x}$ is fixed by $\tau_0$, it suffices to calculate $\mu^{\Mb}(x,\tau_0)$ for $x\in\pi^{-1}(\overline{x})^{\tau_0}$. So we take $x\in\pi^{-1}(\overline{x})^{\tau_0}$.

\vspace{5mm}

\noindent\underline{Form of an element $x\in\pi^{-1}(\overline{x})^{\tau_0}$:} First of all, the action of $\tau_0$ on $V_1$ has three different weights: 1,-1, and 0, whose corresponding subspaces are
\[ W_1=\C e_1, \; W_{-1}=\C e_2,\text{ and }W_0=\C e_3+\dots+\C e_{n_1}. \]
Thus, the component of $x$ in $\fl(V_1)$ is a flag given by a basis of $V_1$ composed of: $e_1$ at first, $e_2$ in a position $i$ between 2 and $n_1$, and $n_1-2$ vectors forming a basis of $W_0$. For the same reasons, there exists an integer $j$ between 2 and $n_2$ such that the second component of $x$ (in $\fl(V_2)$) is a flag given by a basis of $V_2$ composed of $f_2$ at first, $f_1$ in position $j$, and $n_2-2$ vectors forming a basis of $\C f_3+\dots+\C f_{n_2}$.

\vspace{5mm}

For the third component (in $\fl(V_1\otimes V_2)$) of $x$: the action of $\tau_0$ on $V_1\otimes V_2$ has now five different weights, 2, -2, 1, -1, and 0, whose respective corresponding subspaces are
\[ W_2=\C e_1\otimes f_2, \; W_{-2}=\C e_2\otimes f_1, \; W_1=\C e_1\otimes f_3+\dots+\C e_1\otimes f_{n_2}+\C e_3\otimes f_2+\dots+\C e_{n_1}\otimes f_2, \]
\[ W_{-1}=\C e_2\otimes f_3+\dots+\C e_2\otimes f_{n_2}+\C e_3\otimes f_1+\dots+\C e_{n_1}\otimes f_1, \]
\[ W_0\text{ spanned by the rest of the }e_i\otimes f_j. \]
Thus, the component of $x$ in $\fl(V_1\otimes V_2)$ is a flag given by a basis of $V_1\otimes V_2$ of the form:
\begin{itemize}
\item $e_1\otimes f_2$ at a position $k_2$ between 1 and $n_1n_2-1$,
\item $e_2\otimes f_1$ at a position $k_{-2}$ between 1 and $n_1n_2-1$,
\item $n_1+n_2-4$ vectors forming a basis of $W_1$ at positions $m_1^{(1)},\dots,m_{n_1+n_2-4}^{(1)}$ (between 1 and $n_1n_2-1$),
\item $n_1+n_2-4$ vectors forming a basis of $W_{-1}$ at positions $m_1^{(-1)},\dots,m_{n_1+n_2-4}^{(-1)}$ (between 1 and $n_1n_2-1$),
\item the other vectors forming a basis of $W_0$.
\end{itemize}

\vspace{5mm}

\noindent\underline{Calculation of the action of $\tau_0$ on the fibre of $\Mb$ over $x$:} (We denote this fibre by $\Mb_x$).\\
Let us recall another description, for $\delta$ a partition, of the line bundle $\Lb_\delta$ over a flag variety $\fl(V)$ (with $\dim V=n\geq\ell(\delta)$). We have the embedding
\vspace{-5mm}
\begin{changemargin}{-2mm}{-2mm}
\[ \begin{array}{rccl}
\iota: & \fl(V) & \longrightarrow & \prod_{k=1}^n\pp(\bigwedge^k V)\\
 & (\C v_1,\C v_1\oplus\C v_2,\dots,\C v_1\oplus\dots\oplus\C v_n) & \longmapsto & (\C v_1,\C(v_1\wedge v_2),\dots,\C(v_1\wedge\dots\wedge v_n)
\end{array}. \]
\end{changemargin}
Then $\Lb_\delta$ is the pull-back of the line bundle $\oo(\delta_1-\delta_2)\otimes\dots\otimes\oo(\delta_{n-1}-\delta_n)\otimes\oo(\delta_n)$ by $\iota$ (for all the partitions that we use, we take the convention that, if $i>\ell(\delta)$, $\delta_i$ is simply 0). Using this description and the form of an element $x\in\pi^{-1}(\overline{x})^{\tau_0}$, we can easily get the following:

\begin{lemma}
For $x\in\pi^{-1}(\overline{x})^{\tau_0}$, there exist $i\in\llbracket 2,n_1\rrbracket$, $j\in\llbracket 2,n_2\rrbracket$, and $2(n_1+n_2-3)$ distinct integers $k_2,k_{-2},m^{(1)}_1,\dots,$ $m^{(1)}_{n_1+n_2-4},m^{(-1)}_1,\dots,m^{(-1)}_{n_1+n_2-4}\in\llbracket 1,n_1n_2-1\rrbracket$ such that
\[ \mu^\Mb(x,\tau_0)=\lambda_1-\lambda_i+\mu_1-\mu_j+2(\nu'_{k_{-2}}-\nu'_{k_2})+ \sum_{k=1}^{n_1+n_2-4}(\nu'_{m^{(-1)}_k}-\nu'_{m_k^{(1)}}), \]
with $(\nu'_1,\dots,\nu'_{n_1n_2})=(\nu_{n_1n_2},\dots,\nu_1)$. Moreover, all the possibilities for $i,j,k_2,k_{-2}$, the $m^{(1)}_k$'s, and the $m^{(-1)}_k$'s arise when $x$ varies in $\pi^{-1}(\overline{x})^{\tau_0}$.\\
As a consequence,
\[ \max_{x\in\pi^{-1}(\overline{x})}\left(-\mu^\Mb(x,\tau_0)\right)=-\lambda_1+\lambda_2-\mu_1+\mu_2+2(\nu_2-\nu_{n_1n_2})+ \sum_{k=1}^{n_1+n_2-4}(\nu_{k+2}-\nu_{n_1n_2-k}). \]
\end{lemma}

Finally, Lemma \ref{x_bar_suffices} leads to the following result:

\begin{prop}\label{strict_bound_murnaghan}
If we set
\[ d_0=\frac{1}{2}\left(-\lambda_1+\lambda_2-\mu_1+\mu_2+2(\nu_2-\nu_{n_1n_2})+ \sum_{k=1}^{n_1+n_2-4}(\nu_{k+2}-\nu_{n_1n_2-k})\right), \]
we have for all $d\in\N$ such that $d>d_0$,\\
$g_{\lambda+d(1),\mu+d(1),\nu+d(1)}=g_{\lambda+\lfloor d_0+1\rfloor(1),\mu+\lfloor d_0+1\rfloor(1),\nu+\lfloor d_0+1\rfloor(1)}$.
\end{prop}

\begin{proof}
For all $x\in\pi^{-1}(\overline{x})$ and all $d>d_0$,
\[ \mu^{\Mb+d\Lb}(x,\tau_0)=\mu^\Mb(x,\tau_0)+d\mu^{\overline{\Lb}}(\overline{x},\tau_0)=\mu^\Mb(x,\tau_0)+2d>0 \]
because $d>d_0\geq -\dps\frac{1}{2}\mu^\Mb(x,\tau_0)$. Thus, by Hilbert-Mumford criterion, $x\in X^{us}(\Mb+d\Lb)$, and we conclude using Lemma \ref{x_bar_suffices}.
\end{proof}

\begin{rmk}
We even have the inclusion $X^{ss}(\Mb+d\Lb)\subset X^{ss}(\Lb)$ which is true for all $d>d_0$, $d\in\Q$. Indeed, the definition of $X^{ss}(\Nb)$ (and the one from $\mu^\Nb(.,\tau_0)$) can be extended to $\Nb\in\pic^G(X)\otimes_{\Z}\Q$:
\[ X^{ss}(\Nb)=\{x\in X\,|\,\exists k\in\N^*\text{ s.t. }\Nb^{\otimes k}\in\pic^G(X)\text{ and }\exists\sigma\in\h^0(X,\Nb^{\otimes k})^G, \, \sigma(x)\neq 0\}. \] 
\end{rmk}

\subsection[Case of the triple ((1,1),(1,1),(2))]{Case of the triple $\big((1,1),(1,1),(2)\big)$}\label{other_example}

We now have a look at the triple $\big((1,1),(1,1),(2)\big)$ which is also stable (cf. for instance \cite{stembridge}). We consider
\vspace{-5mm}
\begin{changemargin}{-6mm}{-6mm}
\[ \begin{array}{rccl}
\pi: & X & \longrightarrow & \dps\overset{\text{denoted }\overline{X}}{\overbrace{\fl(V_1;1,2)\times\fl(V_2;1,2)\times\pp((V_1\otimes V_2)^*)}}\\
 & ((W_i)_i,(W'_i)_i,(W''_i)_i) & \longmapsto & ((W_1,W_2),(W'_1,W'_2),\{\varphi\in(V_1\otimes V_2)^*\,/\,\ker\varphi=W''_{n_1n_2-1}\})
\end{array}. \]
\end{changemargin}
Similarly as before, the line bundle $\Lb$ is the pull-back by $\pi$ of $\overline{\Lb}=\Lb_\alpha\otimes\Lb_\beta\otimes\oo(2)$.
The same things that we have done throughout the previous section are also going to work here. The only changes will be the orbits of $G$ in $\overline{X}$ which are unstable:

\begin{prop}
If $n_1\geq 3$ or $n_2\geq 3$, then the set $\overline{X}^{us}(\overline{\Lb})$ of unstable points consists in the union of the closures of two orbits: the one of $\overline{x}_1=((\C e_1,\C e_1+\C e_2),(\C f_3,\C f_3+\C f_1),\C\varphi_n)$ and the one of $\overline{x}_2=((\C e_1,\C e_1+\C e_2),(\C f_2,\C f_2+\C f_3),\C\varphi_n)$.
\end{prop}

\begin{proof}
It is completely similar to the proof of Proposition \ref{orbits_murnaghan}.
\end{proof}

We then set two destabilising one-parameter subgroups of $G$ for the two elements $\overline{x}_1$ and $\overline{x}_2$ (we still consider the case when $n_1,n_2\geq 3$):
\[ \tau_1=\left.\big(0,1,-1,0,\dots,0\text{ }\right|\text{ }0,-1,1,0,\dots,0\big) \]
and
\[ \tau_2=\left.\big(1,0,-1,0,\dots,0\text{ }\right|-1,0,1,0,\dots,0\big), \]
which give
\[ \mu^{\overline{\Lb}}(\overline{x}_1,\tau_1)=2=\mu^{\overline{\Lb}}(\overline{x}_2,\tau_2). \]
As before, we only have to get a bound from which $\pi^{-1}(\overline{x}_1)\subset X^{us}(\Mb+d\Lb)$ and $\pi^{-1}(\overline{x}_2)\subset X^{us}(\Mb+d\Lb)$. We have already seen the form of elements of $\pi^{-1}(\overline{x}_1)^{\tau_1}$ and $\pi^{-1}(\overline{x}_2)^{\tau_2}$, and as a consequence we get:

\begin{lemma}
If $n_1\geq 3$ and $n_2\geq 3$,
\[ \max_{x_1\in\pi^{-1}(\overline{x}_1)}(-\mu^\Mb(x_1,\tau_1))=-\lambda_2+\lambda_3-\mu_1+\mu_3+2(\nu_2-\nu_{n_1n_2})+ \sum_{k=1}^{n_1+n_2-4}(\nu_{k+2}-\nu_{n_1n_2-k}) \]
and
\[ \max_{x_2\in\pi^{-1}(\overline{x}_2)}(-\mu^\Mb(x_2,\tau_2))=-\lambda_1+\lambda_3-\mu_2+\mu_3+2(\nu_2-\nu_{n_1n_2})+ \sum_{k=1}^{n_1+n_2-4}(\nu_{k+2}-\nu_{n_1n_2-k}). \]
\end{lemma}

\vspace{5mm}

What remains to be seen is what happens when $n_1=2$ or $n_2=2$. Let us focus on the case where $n_1=2$ and $n_2\geq 3$. Then, $\tau_1$ and $\tau_2$ become:
\[ \tau_1=\left.\big(0,1\text{ }\right|\text{ }0,-1,1,0,\dots,0\big), \qquad \tau_2=\left.\big(1,0\text{ }\right|-1,0,1,0,\dots,0\big). \]
We still have $\mu^{\overline{\Lb}}(\overline{x}_1,\tau_1)=2=\mu^{\overline{\Lb}}(\overline{x}_2,\tau_2)$, but this time
\[ \max_{x_1\in\pi^{-1}(\overline{x}_1)}(-\mu^\Mb(x_1,\tau_1))=-\lambda_2-\mu_1+\mu_3+2\nu_2-\nu_{2n_2}+\sum_{k=1}^{n_2-1}\nu_{k+2}, \]
and
\[ \max_{x_2\in\pi^{-1}(\overline{x}_2)}(-\mu^\Mb(x_2,\tau_2))=-\lambda_1-\mu_2+\mu_3+2\nu_2-\nu_{2n_2}+\sum_{k=1}^{n_2-1}\nu_{k+2}. \]

\vspace{5mm}

By exchanging the roles of $V_1$ and $V_2$ (that is to say $\lambda$ and $\mu$), we easily get the result for the case $n_1\geq 3$, $n_2=2$. Only the case $n_1=2=n_2$ remains, and we could do exactly the same. But the result we would get would be exactly the formula for $n_1=2$, $n_2\geq 3$ in which we take $\mu_3$ to be zero. Finally we have:

\begin{prop}\label{strict_bounds_other_case}
If we set $m=\max(-\lambda_2-\mu_1,-\lambda_1-\mu_2)$ and
\[ d_0=\left\lbrace\begin{array}{ll}
\dps\frac{1}{2}\left(m+\lambda_3+ \mu_3+2(\nu_2-\nu_{n_1n_2})+\dps\sum_{k=1}^{n_1+n_2-4}(\nu_{k+2}-\nu_{n_1n_2-k})\right) & \text{if }n_1,n_2\geq 3\\
\dps\frac{1}{2}\left(m+\mu_3+2\nu_2-\nu_{2n_2}+\dps\sum_{k=1}^{n_2-1} \nu_{k+2}\right) & \text{if }n_1=2\\
\dps\frac{1}{2}\left(m+\lambda_3+2\nu_2-\nu_{2n_1}+\dps\sum_{k=1}^{n_1-1} \nu_{k+2}\right) & \text{if }n_2=2
\end{array}\right., \]
then we have, for all $d\in\N$ such that $d>d_0$,
\[ g_{\lambda+d(1,1),\mu+d(1,1),\nu+d(2)}=g_{\lambda+\lfloor d_0+1\rfloor(1,1),\mu+\lfloor d_0+1\rfloor(1,1),\nu+\lfloor d_0+1\rfloor(2)}. \]
\end{prop}

\begin{proof}
It is exactly in the same way as the proof of Proposition \ref{strict_bound_murnaghan}.
\end{proof}

\begin{rmk}
We can notice that, in the cases where $n_1=2$ or $n_2=2$, we have two possible bounds: the one which concerns only these cases, or the general one, which we can use by considering $\lambda$ (respectively $\mu$) of length 3 by setting $\lambda_3=0$ (respectively $\mu_3=0$). We will come back to this in Remark \ref{rmk_on_bound2}.
\end{rmk}

\begin{rmk}
As after Proposition \ref{strict_bound_murnaghan}, we have also here that the inclusion $X^{ss}(\Mb+d\Lb)\subset X^{ss}(\Lb)$ is true for all $d>d_0$, $d\in\Q$.
\end{rmk}

\subsection{Slight improvement of the previous bounds}\label{improvement}

In Propositions \ref{strict_bound_murnaghan} and \ref{strict_bounds_other_case}, we got an integer or half-integer $d_0$ such that the sequence $(g_{\lambda+d\alpha,\mu+d\beta,\nu+d\gamma})_d$ is constant for all integers strictly greater than $d_0$. We now want to prove that, if this $d_0$ is an integer, this sequence of Kronecker coefficients already stabilises for our bound $d_0$. We need at first, in the following subsection, to expose a well-known result of quasipolynomiality.

\subsubsection{Piecewise quasipolynomial behaviour of the dimension of invariants in an irreducible representation}\label{section_quasipolynomial}

This part of the paper is quite disconnected with the others. The inspiration for the proofs given here is the article \cite{kumar-prasad}, in which the case of $T$-invariants is studied. Note also that the quasipolynomial behaviour of this kind of multiplicities can be seen as a consequence of the work of E. Meinrenken and R. Sjamaar on $[Q,R]=0$ (it is explained in Section 13 of \cite{paradan-vergne}). The following settings concern this subsection and only this one.

\vspace{5mm}

Let $G$ be a complex reductive group, and $H$ be a subgroup of $G$, also reductive. We consider a maximal torus $T$, a Borel subgroup $B$ of $G$ such that $T\subset B$, and the corresponding flag variety $X=G/B$. We denote by $X^*(T)$ the (multiplicative) group of characters of $T$, and by $Q$ and $\Lambda$ respectively the root lattice and the weight lattice. $\Lambda^+$ (resp. $\Lambda^{++}$) denotes the dominant (resp. dominant regular) weights.

\vspace{5mm}

Let us recall that $X^*(T)$ can be embedded as a sublattice of $\Lambda$ (let set $\iota:X^*(T)\hookrightarrow\Lambda$) and that
\[ Q\subset\iota(X^*(T))\subset\Lambda. \]
Set $X^*(T)^+=\iota(X^*(T))\cap\Lambda^+$, and
\[ \begin{array}{rccl}
m: & X^*(T)^+ & \longrightarrow & \N\\
 & \lambda & \longmapsto & \dim V(\lambda)^H
\end{array}, \]
where $V(\lambda)$ is the irreducible $G$-module with highest weight $\lambda$. The result we want to show is that $m$ is piecewise quasipolynomial. For a more precise statement, let us consider $X^*(\R)=\iota(X^*(T))\otimes_{\Z}\R$, $X^*(\R)^+$ the cone spanned by $X^*(T)^+$, and $X^*(\R)^{++}$ the relative interior of this cone.\\
Here we use the more standard definition of semi-stability: if $\Lb$ is a $H$-linearised line bundle over $X$, a point $x\in X$ is said semi-stable (with respect to $\Lb$) if there exist $n\in\N^*$ and $\sigma\in\h^0(X,\Lb^{\otimes n})^H$ such that $\{y\in X\text{ s.t. }\sigma(y)\neq 0\}$ is affine and contains $x$. To avoid confusion with the notion of semi-stability that we use everywhere but in this subsection, we denote by $X^{ss}_{\mathrm{st}}(\Lb)$ the set of these semi-stable points with respect to $\Lb$. Finally let us denote by $C_1,\dots,C_N$ the chambers in $X^*(\R)^{++}$, i.e. the GIT-classes of maximal dimension. Let us recall that the chambers are the relative interiors of convex rational polyhedral cones in $X^*(\R)^{++}$ (see \cite{ressayre2}). For all $k$, denote by $X^{ss}_{\mathrm{st}}(C_k)$ the set of semi-stable points common to all $\Lb_\lambda$ for $\lambda\in C_k$.

\begin{lemma}\label{lemma_lattice}
There exists a sublattice $\Gamma$ of $\iota(X^*(T))$ of finite index such that, for all $k\in\llbracket 1,N\rrbracket$, for all $\lambda\in\Gamma$, the $H$-linearised line bundle $\Lb_\lambda=G\times_B \C_{-\lambda}$ descends to a line bundle on $X^{ss}_{\mathrm{st}}(C_k)\sslash H$ (i.e. the restriction of $\Lb_\lambda$ to $X^{ss}_{\mathrm{st}}(C_k)$ is $H$-isomorphic to the pull-back of a line bundle on $X^{ss}_{\mathrm{st}}(C_k)\sslash H$).
\end{lemma}

\begin{proof} For better readability we have divided this proof into four steps.\\
\underline{First step:} we want to prove that, for all $k\in\llbracket 1,N\rrbracket$, there exists a sublattice $\Gamma_k$ of $\iota(X^*(T))$ of finite index such that, for all $\lambda\in\Gamma_k\cap C_k^\ell$ (where $C_k^\ell=C_k\cap\iota(X^*(T))$), $\Lb_\lambda$ descends to $X^{ss}_{\mathrm{st}}(C_k)\sslash H$.\\
Let $k\in\llbracket 1,N\rrbracket$. We set
\[ A_k=\{\lambda\in C_k^\ell\text{ s.t. }\Lb_\lambda\text{ descends to }X^{ss}_{\mathrm{st}}(C_k)\sslash H\}. \]
Then it is clear that $A_k$ is stable by addition. Thus consider $\Gamma_k$ the lattice generated by $A_k$. It satisfies $A_k=\Gamma_k\cap C_k^\ell$ and so, for all $\lambda\in\Gamma_k\cap C_k^\ell$, $\Lb_\lambda$ descends to $X^{ss}_{\mathrm{st}}(C_k)\sslash H$.\\
Let us now check that $\Gamma_k$ is of finite index in $\iota(X^*(T))$. It suffices to prove that there exists $n\in\N^*$ such that, for all $\lambda\in C_k^\ell$, $n\lambda\in\Gamma_k$, i.e. $\Lb_{n\lambda}\simeq\Lb_\lambda^{\otimes n}$ descends to $X^{ss}_{\mathrm{st}}(C_k)\sslash H$.\\
For all $\lambda\in\overline{C_k^\ell}=\overline{C_k}\cap\iota(X^*(T))$ and $x\in X^{ss}_{\mathrm{st}}(C_k)\subset X^{ss}_{\mathrm{st}}(\Lb_\lambda)$, by definition we know that there exist $n_{x,\lambda}\in\N^*$ and $\sigma_{x,\lambda}\in\h^0(X,\Lb_\lambda^{\otimes n_{x,\lambda}})^H$ such that $\sigma_{x,\lambda}(x)\neq 0$. Let $\lambda\in C_k^\ell$. Then the algebra
\[ R=\bigoplus_{n\geq 0}\h^0(X,\Lb_\lambda^{\otimes n})^H \]
is of finite type. Let us set $\sigma_1,\dots,\sigma_r$ a system of generators of $R$ (we can choose $\sigma_i\in\h^0(X,\Lb_\lambda^{\otimes n_i})^H$ for some $n_i\in\N^*$). Write $n_\lambda=\prod_{i=1}^r n_i\in\N^*$. Then, for $x\in X^{ss}_{\mathrm{st}}(C_k)$, there exists
\[ \sigma_{x,\lambda}=\sigma_1^{\otimes a_1}\otimes\dots\otimes\sigma_r^{\otimes a_r} \]
with $a_1,\dots,a_r\in\N$ not all zero such that $\sigma_{x,\lambda}(x)\neq 0$. So there exists $i\in\llbracket 1,r\rrbracket$ such that $\sigma_i(x)\neq 0$. Hence $\sigma_i^{\otimes n_1\dots n_{i-1}n_{i+1}\dots n_r}(x)\neq 0$ with $\sigma_i^{\otimes n_1\dots n_{i-1}n_{i+1}\dots n_r}=\sigma_0\in\h^0(X,\Lb_\lambda^{\otimes n_\lambda})^H$.\\
Thus, if we denote by $\chi$ the character by which $H_x$ acts on the fiber $\left(\Lb_\lambda^{\otimes n_\lambda}\right)_x$, we have
\[ \forall h\in H_x, \: \chi(h)\sigma_0(x)=h.\sigma_0(x)=\sigma_0(h.x)=\sigma_0(x), \]
and so $\chi$ is trivial. We have just proven that, for all $x\in X^{ss}_{\mathrm{st}}(C_k)$, $H_x$ acts trivially on $\left(\Lb_\lambda^{\otimes n_\lambda}\right)_x$. In other words, by Kempf's Descent Lemma (see e.g. Lemma 3.8 in \cite{kumar}), $\Lb_{n_\lambda\lambda}$ descends to $X^{ss}_{\mathrm{st}}(C_k)$, i.e. $n_\lambda\lambda\in\Gamma_k$.\\
Now, $\overline{C_k^\ell}$ is finitely generated, since it is the intersection of a lattice and a closed convex rational polyhedral cone (see e.g. Section 5.18 from \cite{schrijver} on Hilbert bases). So if we take $\lambda_1,\dots,\lambda_p$ generators, by setting $n=\prod_{i=1}^p n_{\lambda_i}\in\N^*$ we get:
\[ \forall\lambda\in C_k^\ell, \: n\lambda\in\Gamma_k. \]
Thus $\Gamma_k$ is of finite index in $\iota(X^*(T))$.

\vspace{5mm}

\underline{Second step:} Now we set
\[ \Gamma=\bigcap_{k=1}^N\Gamma_k. \]
It is a sublattice of $\iota(X^*(T))$ of finite index, since $\Gamma_1,\dots,\Gamma_N$ are. Moreover, for all $k\in\llbracket 1,N\rrbracket$, for all $\lambda\in\Gamma\cap C_k\subset\Gamma_k\cap C_k$, $\Lb_\lambda$ descends to a line bundle denoted $\hat{\Lb}_\lambda^{(k)}$ on $X^{ss}_{\mathrm{st}}(C_k)$.

\vspace{5mm}

\underline{Third step:} Let $k\in\llbracket 1,N\rrbracket$. We can notice that $\Gamma\cap C_k$ is a semigroup: it is the intersection between a lattice and the interior of a convex rational polyhedral cone. Let us consider $Z_k$ the subgroup of $\Gamma$ generated by $\Gamma\cap C_k$. Let $\lambda\in Z_k$. It can be written as $\lambda_1-\lambda_2$, with $\lambda_1,\lambda_2\in\Gamma\cap C_k$. Then we define
\[ \hat{\Lb}_\lambda^{(k)}=\hat{\Lb}_{\lambda_1}^{(k)}\otimes \left(\hat{\Lb}_{\lambda_2}^{(k)} \right)^*, \]
which is a line bundle over $X^{ss}_{\mathrm{st}}(C_k)\sslash H$. If $\lambda=\lambda'_1-\lambda'_2$ also ($\lambda'_1,\lambda'_2\in\Gamma\cap C_k$), then $\lambda_1+\lambda'_2=\lambda'_1+\lambda_2\in\Gamma\cap C_k$ and so $\hat{\Lb}_{\lambda_1+\lambda'_2}^{(k)}\simeq \hat{\Lb}_{\lambda'_1+\lambda_2}^{(k)}$. Moreover, by uniqueness of the line bundle to which a line bundle can descend (cf. \cite{teleman}, §3), $\hat{\Lb}_{\lambda_1+\lambda'_2}^{(k)}\simeq\hat{\Lb}_{\lambda_1}^{(k)}\otimes \hat{\Lb}_{\lambda'_2}^{(k)}$, and similarly for $\hat{\Lb}_{\lambda'_1+\lambda_2}^{(k)}$. Thus,
\[ \hat{\Lb}_{\lambda_1}^{(k)}\otimes\left(\hat{\Lb}_{\lambda_2}^{(k)} \right)^*\simeq\hat{\Lb}_{\lambda'_1}^{(k)}\otimes \left(\hat{\Lb}_{\lambda'_2}^{(k)} \right)^*, \]
and our $\hat{\Lb}_\lambda^{(k)}$ is well defined. As a consequence $\Lb_\lambda$ descends to a line bundle on $X^{ss}_{\mathrm{st}}(C_k)\sslash H$ for all $\lambda\in Z_k$.

\vspace{5mm}

\underline{Fourth step:} To conclude, let us prove that $Z_k=\Gamma$. We consider $\gamma_1,\dots,\gamma_r$ a system of generators of $\Gamma$, and a norm $\|.\|$ on $X^*(\R)$. Set $d=\max\{\|\gamma_i\|\,;\,i\in\llbracket 1,r\rrbracket\}$. Then there exists $\lambda\in\Gamma\cap C_k$ such that $B(\lambda,d)$, the closed ball of center $\lambda$ and radius $d$, is contained in $C_k$. Hence $\lambda+\gamma_i\in B(\lambda,d)\subset C_k$ for all $i\in\llbracket 1,r\rrbracket$.\\
So, for all $i$, $\lambda+\gamma_i\in\Gamma\cap C_k$, and thus $\gamma_i\in Z_k$. Hence $Z_k=\Gamma$, which proves the lemma.
\end{proof}

The following result is then a classical one. The proof we write here is an adaptation (but with less quantitative results) from the one by Kumar and Prasad in \cite{kumar-prasad}, which was in the case of $T$-invariants.

\begin{theo}\label{theo_quasipolynomial}
Let $\bar{\mu}=\mu+\Gamma$ be a coset of $\Gamma$ in $\iota(X^*(T))$ and $k\in\llbracket 1,N\rrbracket$. Then there exists a polynomial $f_{\bar{\mu},k}:X^*(\R)\rightarrow\R$ with rational coefficients such that,
\[ \forall\lambda\in\overline{C_k}\cap\bar{\mu}, \: m(\lambda)=f_{\bar{\mu},k}(\lambda). \]
\end{theo}

\begin{proof}
Let $\bar{\mu}$ and $k$ be as in the above statement. Applying the Borel-Weil-Bott's Theorem we get that, for all $\lambda\in X^*(T)^+$, $\h^0(X,\Lb_\lambda)\simeq V(\lambda)^*$ and, for all $p>0$, $\h^p(X,\Lb_\lambda)=\{0\}$. As a consequence, since $\dim\left(V(\lambda)^H\right)=\dim\left(\left(V(\lambda)^*\right)^H\right)$,
\[ m(\lambda)=\dim\left(\h^0(X,\Lb_\lambda)^H\right). \]

\vspace{5mm}

Let us begin by considering $\lambda\in C_k\cap\bar{\mu}$. Denote by $\pi$ the standard quotient map $X^{ss}_{\mathrm{st}}(C_k)\rightarrow X^{ss}_{\mathrm{st}}(C_k)\sslash H$ and, for any $H$-equivariant sheaf $\mathcal{S}$ on $X^{ss}_{\mathrm{st}}(C_k)$, by $\pi_*(\mathcal{S})$ the $H$-invariant direct image sheaf of $\mathcal{S}$ by $\pi$ (it is then a sheaf on the GIT-quotient).\\
Then, by \cite{teleman}, Remark 3.3(i),
\[ \h^p(X^{ss}_{\mathrm{st}}(C_k)\sslash H,\pi_*(\Lb_\lambda))\simeq\left\lbrace\begin{array}{ll}
\{0\} & \text{if }p>0\\
\h^0(X,\Lb_\lambda)^H & \text{if }p=0
\end{array}\right.. \]
And thus, if $\chi$ is the Euler-Poincaré characteristic,
\[ \chi(X^{ss}_{\mathrm{st}}(C_k)\sslash H,\pi_*(\Lb_\lambda))=\sum_{p\geq 0}(-1)^p\dim\left(\h^p(X^{ss}_{\mathrm{st}}(C_k)\sslash H,\pi_*(\Lb_\lambda))\right)=m(\lambda). \]

\vspace{5mm}

Take now $\lambda\in\overline{C_k}\cap\bar{\mu}$. We consider $P$ (containing $B$) the unique parabolic subgroup of $G$ such that $\Lb_\lambda$ descends as an ample line bundle $\Lb_\lambda^P$ on $G/P$ via the standard projection $q:X=G/B\rightarrow G/P$. Let $\nu\in C_k\cap\iota(X^*(T))$.\\
Then, by \cite{teleman}, §1.2, for any small enough rational $\varepsilon>0$, the pull-back $q^*(\Lb_\lambda^P)$ is adapted to the stratification on $X$ coming from $q^*(\Lb_\lambda^P)+\varepsilon\Lb_\nu$. So, by \cite{teleman}, Remark 3.3(ii), 
\[ \forall p\in\N, \quad \h^p\left(X^{ss}_{\mathrm{st}}(q^*(\Lb_\lambda^P)+\varepsilon\Lb_\nu)\sslash H,\pi_*(q^*(\Lb_\lambda^P))\right)\simeq\h^p(X,q^*(\Lb_\lambda^P))^H. \]
Moreover, $q^*(\Lb_\lambda^P)=\Lb_\lambda$ and $X^{ss}_{\mathrm{st}}(q^*(\Lb_\lambda^P)+\varepsilon\Lb_\nu)=X^{ss}_{\mathrm{st}}(\Lb_{\lambda+\varepsilon\nu})=X^{ss}_{\mathrm{st}}(C_k)$ (because $\lambda+\varepsilon\nu\in C_k$ if $\varepsilon$ is small enough), and thus
\[ \forall p\in\N, \quad \h^p(X^{ss}_{\mathrm{st}}(C_k)\sslash H,\pi_*(\Lb_\lambda))\simeq\h^p(X,\Lb_\lambda)^H. \]
Consequently we have once again
\[ m(\lambda)=\chi(X^{ss}_{\mathrm{st}}(C_k)\sslash H,\pi_*(\Lb_\lambda)). \]

\vspace{5mm}

We now introduce a $\Z$-basis $(\gamma_1,\dots,\gamma_r)$ of the lattice $\Gamma$. For any $\lambda=\mu+\sum_{i=1}^r a_i\gamma_i\in\overline{C_k}\cap\bar{\mu}$ (i.e. with $a_1,\dots,a_r\in\Z$),
\[ \pi_*(\Lb_\lambda)\simeq\pi_*(\Lb_\mu)\otimes\hat{\Lb}_{a_1\gamma_1+\dots+ a_r\gamma_r}^{(k)} \]
by definition of the lattice $\Gamma$ and the projection formula for $\pi_*$, and with the notation $\hat{\Lb}$ defined in the proof of Lemma \ref{lemma_lattice}. Finally, for any such $\lambda$ we apply the Riemann-Roch Theorem for singular varieties (see e.g. \cite{fulton2}, Theorem 18.3), to the sheaf $\pi_*(\Lb_\lambda)$ and get
\[ \begin{array}{rcl}
m(\lambda) & = & \chi(X^{ss}_{\mathrm{st}}(C_k)\sslash H,\pi_*(\Lb_\lambda))\\
 & = & \dps\sum_{n\geq 0}\dps\int_{X^{ss}_{\mathrm{st}}(C_k)\sslash H}\dps\frac{(a_1c_1(\gamma_1)+\dots+a_rc_1(\gamma_r))^n}{n!}\cap\tau(\pi_*(\Lb_\mu)),
\end{array} \]
where, for all $i$, $c_1(\gamma_i)$ is the first Chern class of the line bundle $\hat{\Lb}_{\gamma_i}^{(k)}$, and $\tau(\pi_*(\Lb_\mu))$ is a certain class in the Chow group $A_*(X^{ss}_{\mathrm{st}}(C_k)\sslash H)\otimes_\Z \Q$. Hence $m(\lambda)$ is a polynomial with rational coefficients in the variables $a_i$.
\end{proof}

\subsubsection[Improvement of the bounds of Sections 3.2 and 3.3]{Improvement of the bounds of Sections \ref{example_murnaghan} and \ref{other_example}}

We now come back to the notations of Section \ref{general_section}.

\begin{prop}
If $d_0\in\N$ is such that, for all $d\in\Q$ such that $d>d_0$, $X^{ss}(\Mb+d\Lb)\subset X^{ss}(\Lb)$, then
\[ \dim\left(\h^0(X,\Mb+d_0\Lb)^G\right)=\dim\left(\h^0(S,\Mb)^H\right). \]
\end{prop}

\begin{proof}
Let us write $\ell=\dim\left(\h^0(S,\Mb)^H\right)$, consider a $d_0\in\N$ as in the statement above, and denote by $C_1,\dots,C_N$ the chambers (i.e. GIT-classes of maximal dimension) in $\Q\Lb\oplus\Q\Mb$ for the action of $G$ on $X$. Since, for all $d>d_0$ ($d\in\Q$), $X^{ss}(\Mb+d\Lb)\subset X^{ss}(\Lb)$, and thanks to the results by N. Ressayre (cf. \cite{ressayre2}) concerning the GIT-fan, the situation is necessarily the following:
\begin{itemize}
\item the $\Mb+d\Lb$ for $d>d_0$ are in a chamber, say for instance $C_1$;
\item $\Lb$ belongs to $\overline{C_1}$, the closure of this chamber;
\item $\Mb+d_0\Lb$ belongs also to $\overline{C_1}$.
\end{itemize}
We can draw a picture of this situation: in $\Q\Lb\oplus\Q\Mb$, the set of semi-ample line bundles is a closed convex cone. As a consequence, up to multiplication by a positive rational number, this set can be represented by a line or a segment. The two cases can here be treated in the same way, so we assume for instance to be in the case of a line. Then the situation of the chambers is typically:
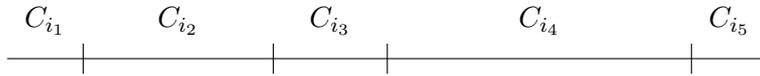
\begin{center}
\begin{tikzpicture}
\draw (0,0)--(10,0);
\draw (1,0.2)--(1,-0.2);
\draw (3.5,0.2)--(3.5,-0.2);
\draw (5,0.2)--(5,-0.2);
\draw (9,0.2)--(9,-0.2);
\node at (0.5,0.5) {$C_{i_1}$};
\node at (2.25,0.5) {$C_{i_2}$};
\node at (4.25,0.5) {$C_{i_3}$};
\node at (7,0.5) {$C_{i_4}$};
\node at (9.5,0.5) {$C_{i_5}$};
\end{tikzpicture}
\end{center}
If $\Mb+d_0\Lb\in C_1$, then $X^{ss}(\Mb+d_0\Lb)\subset X^{ss}(\Lb)$ and Proposition \ref{inclusion_of_ss_points_suffices} gives immediately that $\dim\h^0(X,\Mb+d_0\Lb)^G=\ell$. So we assume from now on that $\Mb+d_0\Lb\in\overline{C_1}\setminus C_1$:
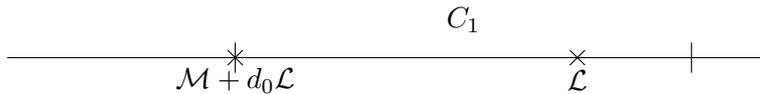
\begin{center}
\begin{tikzpicture}
\draw (0,0)--(10,0);
\draw (3,0.2)--(3,-0.2);
\draw (9,0.2)--(9,-0.2);
\node at (6,0.5) {$C_1$};
\draw (7.4,0.1)--(7.6,-0.1);
\draw (7.4,-0.1)--(7.6,0.1);
\node at (7.5,-0.3) {$\Lb$};
\draw (2.9,0.1)--(3.1,-0.1);
\draw (2.9,-0.1)--(3.1,0.1);
\node at (3,-0.3) {$\Mb+d_0\Lb$};
\end{tikzpicture}
\end{center}
(if $\Lb$ belongs to the boundary of $C_1$, this does not change what follows). 

\vspace{5mm}

Applying Lemma \ref{lemma_lattice} and Theorem \ref{theo_quasipolynomial}, we get that there exists a sublattice $\Gamma$ of $\Z\Lb\oplus\Z\Mb$ of finite index such that, for all $\bar{\gamma}=\gamma+\Gamma\in(\Z\Lb\oplus\Z\Mb)/\Gamma$, there is a polynomial $P_{\bar{\gamma}}$ with rational coefficients such that
\[ \forall\Nb\in\overline{C_1}\cap(\gamma+\Gamma), \: \dim\left(\h^0(X,\Nb)^G\right)=P_{\bar{\gamma}}(\Nb). \]
In particular, if we denote $\bar{\gamma}_0=(\Mb+d_0\Lb)+\Gamma$,
\[ \dim\left(\h^0(X,\Mb+d_0\Lb)^G\right)=P_{\bar{\gamma}_0}(\Mb+d_0\Lb). \]
We then consider the polynomial function in one variable
\[ \tilde{P}:d\longmapsto P_{\bar{\gamma}_0}(\Mb+d\Lb). \]
We want to prove that $\tilde{P}$ is constant and we know that, for all integers $d>d_0$ such that $(\Mb+d\Lb)+\Gamma=\bar{\gamma}_0$, $\tilde{P}(d)=\dim\left(\h^0(X,\Mb+d\Lb)^G\right)=\ell$. It is consequently sufficient to notice that there exist infinitely many such $d$'s: if we denote by $\Nb_1=a_1\Mb+b_1\Lb$ and $\Nb_2=a_2\Mb+b_2\Lb$ the elements of a $\Z$-basis of $\Gamma$, each $d\in\N(b_1a_2-a_1b_2)+d_0$ does the trick. Finally, $\tilde{P}$ is constant and $\dim\left(\h^0(X,\Mb+d_0\Lb)^G\right)=\tilde{P}(d_0)=\ell$.
\end{proof}

Thanks to this result we can improve slightly Propositions \ref{strict_bound_murnaghan} and \ref{strict_bounds_other_case}, and get Theorems \ref{thm_bound1} and \ref{thm_bound2}.

\begin{rmk}\label{rmk_on_bound2}
We had previously noticed that, in the case of triple $\big((1,1),(1,1),(2)\big)$, if $n_1=2$ (or $n_2=2$), there was two ways to compute a bound:
\begin{itemize}
\item by using the formula in the previous theorem which is special to this case,
\item by using the formula valid for $n_1,n_2\geq 3$, setting $\lambda_3=0$ (or $\mu_3=0$) and considering $\lambda$ (or $\mu$) as a partition of length 3.
\end{itemize}
Let us compare the two bounds we can obtain. For instance for three partitions of the form $(\lambda_1,\dots,\lambda_{n_1})$, $(\mu_1,\mu_2)$, and $(\nu_1,\dots,\nu_{2n_1})$ (with $n_1\geq 3$), we obtain by the first method:
\[ D_2=\left\lceil\frac{1}{2}(m+\lambda_3+2\nu_2-\nu_{2n_1}+\nu_3+\nu_4+\dots+\nu_{n_1+1})\right\rceil. \]
And by the second method we get:
\[ D'_2=\left\lceil\frac{1}{2}(m+\lambda_3+2\nu_2+\nu_3+\nu_4+\dots+\nu_{n_1+1})\right\rceil. \]
So we have $D_2\leq D'_2$ and $D'_2-D_2=\lfloor\frac{\nu_{2n_1}}{2}\rfloor$. Similarly, for $(\lambda_1,\lambda_2)$, $(\mu_1,\mu_2)$, and $(\nu_1,\dots,\nu_4)$,
\[ D_2=\left\lceil\frac{1}{2}(m+2\nu_2-\nu_4+\nu_3)\right\rceil,\text{ whereas }D'_2=\left\lceil\frac{1}{2}(m+2\nu_2+\nu_3+\nu_4)\right\rceil. \]
Once again, $D_2\leq D'_2$. And, this time, $D'_2-D_2=\nu_4$.\\
As a conclusion, it is better to use the first way of computing the bound, and that is what we do later on the examples.
\end{rmk}

\subsection{Possibility of recovering already existing bounds by our method}

In the case of Murnaghan's stability, there are some already existing bounds for the stabilisation of the sequence (see the Introduction). An interesting fact is that we can recover (and sometimes improve) some of them by our method, if we choose one-parameter subgroups different from the one that we had chosen. We focus only on two of the four bounds we cited: Brion's one (denoted by $D_B$), and the second one from Briand, Orellana, and Rosas, which we denote by $D_{BOR2}$. They are the ones who have a form similar to our bound; the two other ones seem far too different to be obtained this way.

\subsubsection{Conversion to our settings}

In the article \cite{briand-orellana-rosas}, the settings are different from ours. So, if we want to recover the bounds given here, the first thing is to convert them into our settings. For the authors, the bound given (for a triple of partitions $(\alpha,\beta,\gamma)$) is the first integer $n$ for which $\overline{\alpha}[n]=(n-|\alpha|,\alpha_1,\dots,\alpha_{\ell(\alpha)})$, $\overline{\beta}[n]$, $\overline{\gamma}[n]$ are partitions and the sequence
\[ (g_{\overline{\alpha}[n],\overline{\beta}[n],\overline{\gamma}[n]})_n \]
reaches its limit value (we know that it is a stationary sequence). Whereas for us, our bound for a triple $(\lambda,\mu,\nu)$ of partitions (such that $|\lambda|=|\mu|=|\nu|$) is the first integer $d$ such that the sequence
\[ (g_{\lambda+(d),\mu+(d),\nu+(d)})_d \]
reaches its limit value.\\
The correspondence between the two points of view is then (we adopt the following useful notation: for a partition $\delta$, $\delta_{\geq 2}$ denotes the partition obtained by removing the first -i.e. biggest- part of $\delta$):
\[ \left\lbrace\begin{array}{l}
\alpha=\lambda_{\geq 2}\\
\beta=\mu_{\geq 2}\\
\gamma=\nu_{\geq 2}\\
n=d+|\lambda|=d+|\mu|=d+|\nu|
\end{array}\right.. \]
M. Brion's bound, which in \cite{briand-orellana-rosas} notations is $M_B(\alpha,\beta;\gamma)=|\alpha|+|\beta|+\gamma_1$, then becomes 
\[ D_B(\lambda,\mu,\nu)=|\mu|-\lambda_1-\mu_1+\nu_2. \]
Similarly the bound $D_{BOR2}$, which in their notations is
\[ N_2(\alpha,\beta,\gamma)=\left\lfloor\dps\frac{|\alpha|+|\beta|+|\gamma|+\alpha_1+\beta_1+\gamma_1}{2}\right\rfloor, \]
becomes
\[ D_{BOR2}(\lambda,\mu,\nu)=\left\lfloor\frac{-\lambda_1+|\mu_{\geq 2}|-\nu_1+\lambda_2+\mu_2+\nu_2}{2}\right\rfloor. \]

\subsubsection{One parameter subgroups corresponding to $D_B$ and $D_{BOR2}$}

\underline{Case of $D_B$:} We define the following one parameter subgroup of $G$:
\[ \tau_B=\left.\big(1,0,\dots,0\text{ }\right|-1,0,-1,\dots,-1\big). \]
Thus $\tau_B$ satisfies $\mu^{\overline{\Lb}}(\overline{x},\tau_B)=1$ and, for all $x\in\pi^{-1}(\overline{x})$,
\[ \begin{array}{rcl}
\mu^{\Mb+d\Lb}(x,\tau_B)>0\Longleftrightarrow d & > & \max_{x\in\pi^{-1}(\overline{x})}(-\mu^{\Mb}(x,\tau_B))\\
 & = & -\lambda_1+\mu_2+\mu_3+\dots+\mu_{n_2}+\nu_2-\nu_{n_1+n_2}-\dots-\nu_{n_1n_2}.
\end{array} \]
Until now, we did not make any particular assumption on the flag varieties we considered. We had always taken complete ones, but we could also consider partial ones. Here, let us consider the partial flag variety $\fl(V_1\otimes V_2;1,2,\dots,n_1+n_2-1)$ for the third factor of $X$. This corresponds to forgetting the terms $-\nu_{n_1+n_2}\dots-\nu_{n_1n_2}$ in the right-hand side of the inequality above. This way, this right-hand side is just $D_B(\lambda,\mu,\nu)$. Hence the bound $D_B$ can be recovered by our method, with the one-parameter subgroup $\tau_B$.

\begin{rmk}\label{rmk_improvement_D_B}
We can thus have an improvement of $D_B$ in the case of a ``long'' partition $\nu$: if we keep on with complete flag varieties, we keep the terms $-\nu_{n_1+n_2}\dots-\nu_{n_1n_2}$ at the end of the bound, and so it gives a lower value (and then better one) for partitions $\nu$ of length at least $n_1+n_2$.
\end{rmk}

\vspace{5mm}

\underline{Case of B-O-R 2:} We define the following one parameter subgroup of $G$:
\[ \tau_{BOR2}=\left.\big(1,-1,0,\dots,0\text{ }\right|-2,0,-1,\dots,-1\big). \]
This $\tau_{BOR2}$ satisfies $\mu^{\overline{\Lb}}(\overline{x},\tau_{BOR2})=2$ and, for all $x\in\pi^{-1}(\overline{x})$,
\[ \begin{array}{rcl}
\mu^{\Mb+d\Lb}(x,\tau_{BOR2})>0\Longleftrightarrow d & > & \dps\frac{1}{2}\max_{x\in\pi^{-1}(\overline{x})}(-\mu^{\Mb}(x,\tau_{BOR2}))\\
 & = & \dps\frac{1}{2}(-\lambda_1+\lambda_2+2\mu_2+|\mu_{\geq 3}|-\nu_1+\nu_2-\nu_{n_1+n_2-1}\\
 & & -\dots-\nu_{n_1n_2-n_1-n_2+2}-2(\nu_{n_1n_2-n_1-n_2+3}+\dots\\
 & & +\nu_{n_1n_2-1})-3\nu_{n_1n_2}).
\end{array} \]
Once again, considering the partial flag variety $\fl(V_1\otimes V_2;1,2,\dots,n_1+n_2-2)$ (slightly different from the previous case), we can ``forget'' the terms concerning the last parts of partition $\nu$ (i.e. $-\nu_{n_1+n_2-1}-\dots-\nu_{n_1n_2-n_1-n_2+2}-2(\nu_{n_1n_2-n_1-n_2+3}+\dots+\nu_{n_1n_2-1})-3\nu_{n_1n_2}$) and thus recognise $D_{BOR2}(\lambda,\mu,\nu)$ in the right-hand side of the previous inequality. Hence the bound $D_{BOR2}$ can be recovered by our method, with the one-parameter subgroup $\tau_{BOR2}$.

\begin{rmk}\label{rmk_improvement_D_BOR2}
As for $D_B$, we can also have an improvement of $D_{BOR2}$ by keeping the complete flag variety $\fl(V_1\otimes V_2)$: if $\ell(\nu)\geq n_1+n_2-1$, our method gives a lower bound.
\end{rmk}

\subsection{Tests of our bounds and comparison with existing results}\label{tests}

\subsubsection[Tests and comparison for ((1),(1),(1))]{Tests and comparison for $\big((1),(1),(1)\big)$}\label{test1}

We are now going to test the bound $D_1$ from Theorem \ref{thm_bound1} on a dozen examples. We also compare it to the four other bounds exposed in \cite{briand-orellana-rosas} (Vallejo's bound is denoted by $D_V$, and the first one from Briand, Orellana, and Rosas by $D_{BOR1}$).

The following array presents the results of these bounds on chosen examples. We also added a column giving the minimal integer coming from all the bounds obtainable by our method: ours, $D_B$ (a little improved, by Remark \ref{rmk_improvement_D_B}), and $D_{BOR2}$ (likewise, cf. Remark \ref{rmk_improvement_D_BOR2}). We denote this by $D_m$. Finally, we calculated with Sage\footnote{http://www.sagemath.org/} the first integer -denoted $D_{\text{real}}$- from which the sequence $(g_{\lambda+d(1),\mu+d(1),\nu+d(1)})_{d\in\N}$ actually stabilises.
\[ \begin{array}{c|c|c|c|c|c|c|c}
\text{triple }\lambda,\mu,\nu & D_1 & D_m & D_{\text{real}} & D_B & D_V & D_{BOR1} & D_{BOR2}\\
\hline
(8,5,2),(6,5,2,2),(4,4,3,3,1) & 6 & 5 & 5 & 5 & 5 & 5 & 6\\
(4,3,3),(3,2^3,1),(2^3,1^4) & 4 & 4 & 3 & 5 & 5 & 4 & 4\\
(5,5,4,4),(6^3),(3,3,2^4,1^4) & 5 & 5 & 5 & 10 & 11 & 6 & 9\\
(6,5,5),(8,8),(4,4,3,3,2) & 4 & 4 & 4 & 6 & 7 & 4 & 7\\
(5^4),(4^5),(2^4,1^{12}) & 5 & 4 & 4 & 13 & 14 & 6 & 10\\
(6^3),(3^6),(2^6,1^6) & 7 & 6 & 6 & 11 & 11 & 7 & 9\\
(5,5,4,4),(6^3),(3,2^6,1^3) & 4 & 4 & 4 & 9 & 11 & 5 & 8\\
(7,6),(6,5,2),(7,3,2,1) & 3 & 3 & 3 & 3 & 4 & 3 & 3\\
(8,4,3,3,1),(7,3^4),(14,3,2) & 0 & 0 & 0 & 0 & 0 & 0 & 0\\
(8,5,3,1),(2,1^{15}),(4,3,3,2,2,1^3) & 3 & 1 & 1 & 6 & 7 & 2 & 6\\
(6,6,4),(8,8),(5,5,4,1,1) & 7 & 6 & 6 & 7 & 7 & 7 & 8\\
(8,6,6,2,1),(14,5,4),(5^4,3) & 6 & 6 & 5 & 6 & 8 & 5 & 6
\end{array} \]

We can notice (see e.g. the third row in the array) that there exist cases in which our bound is optimal whereas the other known bounds compared here are not. Ours is of course not always better: see e.g. the last row.

\subsubsection[Tests of the bound for ((1,1),(1,1),(2))]{Tests of the bound for $\big((1,1),(1,1),(2)\big)$}

Here we compute the bound $D_2$ from Theorem \ref{thm_bound2} for a dozen examples and compare it, in the following array, to the first integer $D_{\text{real}}$ from which the sequence actually stabilises. This last integer was once again computed with Sage.
\[ \begin{array}{c|c|c|c|c}
\lambda & \mu & \nu & D_2 & D_{\text{real}}\\
\hline
(5,5,4,4) & (6^3) & (3,3,2^4,1^4) & 5 & 4\\
(5^4) & (4^5) & (2^4,1^{12}) & 5 & 4\\
(6,5,5) & (6,5,5) & (3,3,2^4,1,1) & 4 & 4\\
(8,5,2) & (6,5,2,2) & (4,4,3,2,2) & 4 & 4\\
(4,3,3) & (4,3,3) & (2^3,1^4) & 3 & 3\\
(5,4,4) & (5,4,4) & (3,2^3,1^4) & 3 & 3\\
(6,5,5) & (8,8) & (4,4,3,3,2) & 3 & 2\\
(6,6,6) & (9,9) & (6,4,3,3,2) & 3 & 1\\
(10,8,6) & (12,12) & (6,5,4,4,3,2) & 1 & 1\\
(8,2) & (6,4) & (5,4,1) & 1 & 1\\
(6,6) & (8,4) & (6,4,2) & 0 & 0\\
(20,5) & (13,12) & (11,10,3,1) & 2 & 1
\end{array} \]

\section{Application to plethysm coefficients}\label{plethysm}

The aim of this section is to adapt the techniques we used on Kronecker coefficients to the plethysm coefficients, introduced by J. Littlewood in 1950.

\subsection{Definition and some known stability properties}

We still denote by $\Sc$ the Schur functor. For any partition $\lambda$, we also denote by $n_\lambda$ the dimension of the representation $\Sc^\lambda(V)$. By Weyl Dimension Formula, if $\ell(\lambda)\leq\dim(V)$,
\[ n_\lambda=\prod_{1\leq i<j\leq\ell(\lambda)}\frac{\lambda_i-\lambda_j+j-i}{j-i} \]
(see e.g. \cite{goodman-wallach}). The difficult problem of the composition of Schur functors gives rise to the plethysm coefficients.

\begin{de}\label{def_plethysm}
Let $\lambda$ and $\mu$ be partitions such that $\ell(\lambda)\leq n_\mu$ and $V$ a complex vector space such that $n=\dim V\geq\ell(\mu)$. Then $\Sc^\lambda(\Sc^\mu(V))$ is a representation of $\gl(V)$ and thus splits as a direct sum of irreducible ones:
\[ \Sc^\lambda(\Sc^\mu(V))=\bigoplus_{\nu\text{ s.t. }\ell(\nu)\leq n}a_{\lambda,\mu}^\nu\Sc^\nu(V). \]
The coefficients $a_{\lambda,\mu}^\nu$ are called the plethysm coefficients.
\end{de}

\begin{rmk}
There is a necessary condition (known since the work of Littlewood) on the sizes of the partitions for these coefficients to be non zero: if $|\lambda|.|\mu|\neq|\nu|$, then $a_{\lambda,\mu}^\nu=0$.
\end{rmk}

\vspace{5mm}

There exist for those coefficients some stability properties similar to the ones we studied concerning Kronecker coefficients. The following four are for example given in \cite{colmenarejo}:

\begin{prop}\label{prop_colmenarejo}
For any partitions $\lambda$, $\mu$, and $\nu$, such that $|\lambda|.|\mu|=|\nu|$, the following four sequences of plethysm coefficients are constant for $n$ sufficiently large:
\begin{enumerate}
\item $(a_{\lambda+(n),\mu}^{\nu+(|\mu|n)})_n$,
\item $(a_{\lambda+(n),\mu}^{\nu+n\mu})_n$,
\item $(a_{\lambda,\mu+(n)}^{\nu+(|\lambda|n)})_n$,
\item $(a_{\lambda,\mu+n\pi}^{\nu+n|\lambda|\pi)})_n$ for any partition $\pi$.
\end{enumerate}
Furthermore, the first one has limit zero when $\ell(\mu)>1$, and the second and fourth are non-decreasing.
\end{prop}

\subsection{Link with invariant sections of line bundles}

Starting from Definition \ref{def_plethysm} we get, thanks to Schur's Lemma:
\[ a_{\lambda,\mu}^\nu=\dim(\Sc^\lambda(\Sc^\mu(V))\otimes(\Sc^\nu V)^*)^G \]
(denoting $\gl(V)$ by $G$). Then, Borel-Weil's Theorem gives
\[ (\Sc^\nu V)^*\simeq\h^0(\fl(V),\Lb_\nu) \]
and
\[ \Sc^\lambda(\Sc^\mu(V))\simeq\h^0(\fl(\Sc_\mu(V)),\Lb_\lambda^*). \]
Let us keep in mind that, as a vector space, $\Sc^\mu(V)$ is simply $\C^{n_\mu}$. So we obtain the following proposition:

\begin{prop}
If $V$ is a complex vector space of dimension $n$ and $\lambda$, $\mu$, $\nu$ are three partitions such that $\ell(\lambda)\leq n_\mu$, $\ell(\mu)\leq n$, and $\ell(\nu)\leq n$, then
\[ a_{\lambda,\mu}^\nu=\dim\left(\h^0(X_\mu,\Lb_{\lambda,\nu})^G\right), \]
where $G=\gl(V)$, $X_\mu=\fl(V)\times\fl(\C^{n_\mu})$, and $\Lb_{\lambda,\nu}=\Lb_\nu\otimes\Lb_\lambda^*$.
\end{prop}

For instance, it gives interesting things for two of the sequences cited earlier:
\[ a_{\lambda+(d),\mu}^{\nu+d\mu}=\dim\left(\h^0(X_\mu,\Lb_{\lambda,\nu}+d\Lb_{(1),\mu})^G\right) \]
and
\[ a_{\lambda+(d),\mu}^{\nu+(d|\mu|)}=\dim\left(\h^0(X_\mu,\Lb_{\lambda,\nu}+d\Lb_{(1),(|\mu|)})^G\right). \]
As, in these cases, the projective variety $X_\mu$ does not depend on $d$, we can apply our techniques. For comparison, for the two other sequences cited, it would give a variety depending on $d$ and so it would be a lot different.

\vspace{5mm}

More generally, we are going to consider sequences of general term
\[ a_{\lambda+d\alpha,\mu}^{\nu+d\gamma}=\dim\left(\h^0(X_\mu,\Lb_{\lambda,\nu}+d\Lb_{\alpha,\gamma})^G\right), \]
where $\alpha$ and $\gamma$ are partitions such that $|\alpha|.|\mu|=|\gamma|$.

\subsection{Application of our previous techniques}

Using exactly the same method as for Kronecker coefficients, we get the following result (Sam and Snowden obtained the same in \cite{sam-snowden} by completely different methods, whereas Paradan reproved it in \cite{paradan}):

\begin{theo}\label{main_thm_plethysm}
Let $\lambda$, $\mu$, $\nu$ and $\alpha$, $\gamma$ be partitions such that $|\lambda|.|\mu|=|\nu|$ and, for all $d\in\N^*$, $a_{d\alpha,\mu}^{d\gamma}=1$. Then the sequence $\left(\dps a_{\lambda+d\alpha,\mu}^{\nu+d\gamma}\right)_{d\in\N}$ is non-decreasing and stabilises for $d$ large enough.
\end{theo}

\begin{proof}
The fact that this sequence is non-decreasing is, as in the case of Kronecker coefficients, quite easy: let $\sigma_0\in\h^0(X_\mu,\Lb_{\alpha,\gamma})^G\setminus\{0\}$ (such a section exists because $a_{\alpha,\mu}^\gamma=1$). Then, for all $d\in\N$, we have the following injection:
\[ \begin{array}{rccl}
\iota_d: & \h^0(X_\mu,\Lb_{\lambda,\nu}+d\Lb_{\alpha,\gamma})^G & \longrightarrow & \h^0(X_\mu,\Lb_{\lambda,\nu}+(d+1)\Lb_{\alpha,\gamma})^G\\
 & \sigma & \longmapsto & \sigma\otimes\sigma_0
\end{array}, \]
and thus $a_{\lambda+d\alpha,\mu}^{\nu+d\gamma}\leq a_{\lambda+(d+1)\alpha,\mu}^{\nu+(d+1)\gamma}$.

\vspace{5mm}

For the fact that it stabilises, since it is exactly the same method as for Kronecker coefficients, we are not going to write every details. But here are the principal steps of the proof. First of all,
\[ \begin{array}{rcl}
\h^0(X_\mu,\Lb_{\lambda,\nu}+d\Lb_{\alpha,\gamma})^G & \simeq & \h^0(X_\mu^{ss}(\Lb_{\lambda,\nu}+d\Lb_{\alpha,\gamma}),\Lb_{\lambda,\nu}+d\Lb_{\alpha,\gamma})^G\\
 & \simeq & \h^0(X_\mu^{ss}(\Lb_{\alpha,\gamma}),\Lb_{\lambda,\nu}+d\Lb_{\alpha,\gamma})^G\qquad\text{for }d\gg 0
\end{array} \]
(because $X_\mu^{ss}(\Lb_{\lambda,\nu}+d\Lb_{\alpha,\gamma})\subset X_\mu^{ss}(\Lb_{\alpha,\gamma})$ for $d\gg 0$). Then, since $\h^0(X_\mu,\Lb_{\alpha,\gamma}^{\otimes d})^G\simeq\C$ for all $d\in\N^*$ and using Luna's Slice \'{E}tale Theorem,
\[ \begin{array}{rcl}
\h^0(X_\mu,\Lb_{\lambda,\nu}+d\Lb_{\alpha,\gamma})^G & \simeq & \h^0(G\times_H S,\Lb_{\lambda,\nu}+d\Lb_{\alpha,\gamma})^G\\
 & \simeq & \h^0(S,\Lb_{\lambda,\nu}+d\Lb_{\alpha,\gamma})^H
\end{array} \]
(notations are the same as in the Kronecker coefficients' case). Finally, we have also here that the line bundle $\Lb_{\alpha,\gamma}$ is trivial on $S$. Thus
\[ \h^0(X_\mu,\Lb_{\lambda,\nu}+d\Lb_{\alpha,\gamma})^G\simeq\h^0(S,\Lb_{\lambda,\nu})^H\qquad\text{for }d\gg 0. \]
\end{proof}

\vspace{5mm}

This theorem applies to one of the examples given above: the sequence $(a_{\lambda+(d),\mu}^{\nu+d\mu})_{d\in\N}$. To see that, one just has to check that, for all $d\in\N^*$, $a_{(d),\mu}^{d\mu}=1$. Let us set $d\in\N^*$. The coefficient $a_{(d),\mu}^{d\mu}$ is by definition the multiplicity of the irreducible representation $\Sc^{d\mu}(V)$ in the decomposition of $\sym^d(\Sc^\mu(V))$ ($\sym$ denotes the symmetric power).
\begin{itemize}
\item First, if $v\in\Sc^\mu(V)$ is of weight $\mu$ (denoted $v\in\Sc^\mu(V)_\mu$), then $v^d\in\sym^d(\Sc^\mu(V))$ is of weight $d\mu$. So $\dim\left(\sym^d(\Sc^\mu(V))_{d\mu}\right)\geq 1$.
\item Moreover, $\dim\Sc^\mu(V)_\mu=1$ and the set of weights in $\Sc^\mu(V)$ is $\wt(\Sc^\mu(V))=\{\mu\}\sqcup\{\text{weights}<\mu\}$. So, since a well-known (and easy to understand) fact is that the weights of $\sym^d(\Sc^\mu(V))$ are among $\{\chi_1+\dots+\chi_d\;;\;\chi_1,\dots,\chi_d\in\wt(\Sc^\mu(V))\}$, $\dim\left(\sym^d(\Sc^\mu(V))_{d\mu}\right)=1$.
\item Finally, $\wt(\sym^d(\Sc^\mu(V)))\subset\{\chi_1+\dots+\chi_d\;;\;\chi_1,\dots,\chi_d\in\wt(\Sc^\mu(V))\}$ also gives us that $\wt(\sym^d(\Sc^\mu(V)))=\{d\mu\}\sqcup\{\text{weigths}<d\mu\}$.
\end{itemize}
Thus we have $a_{(d),\mu}^{d\mu}=1$.

\subsection[Other example, where Theorem 4.5 does not apply]{Other example, where Theorem \ref{main_thm_plethysm} does not apply}

Now what about the other sequence cited as example: $(a_{\lambda+(d),\mu}^{\nu+(d|\mu|)})_{d\in\N}$? When $\ell(\mu)=1$, it is the same as before. So assume $\ell(\mu)>1$.

\vspace{5mm}

Let us set $d\in\N^*$ and compute $a_{(d),\mu}^{(d|\mu|)}$. This coefficient is the multiplicity of $\sym^{d|\mu|}(V)$ inside $\sym^d(\Sc^\mu(V))$. If $\sym^{d|\mu|}(V)$ appears in $\sym^d(\Sc^\mu(V))$, then there exist vectors of weight $(d|\mu|)$ in $\sym^d(\Sc^\mu(V))$. But we already explained what weights of $\sym^d(\Sc^\mu(V))$ look like. So, if $\sym^{d|\mu|}(V)$ appears in $\sym^d(\Sc^\mu(V))$, then $(d|\mu|)=\chi_1+\dots+\chi_d$ with $\chi_1,\dots,\chi_d\in\wt(\Sc^\mu(V))$. Then, for all $i\in\llbracket 1,d\rrbracket$, $\chi_i=(|\mu|)$. But $(|\mu|)$ is not a weight of $\Sc^\mu(V)$ (because $\ell(\mu)>1$ and the weights of $\Sc^\mu(V)$ are in the convex hull of $W.\mu$, where $W$ denotes the Weyl group of $G$). Thus $\sym^{d|\mu|}(V)$ does not appear in $\sym^d(\Sc^\mu(V))$, which means that $a_{(d),\mu}^{(d|\mu|)}=0$.

\vspace{5mm}

As a consequence, $X_\mu^{ss}(\Lb_{(1),(|\mu|)})=\emptyset$ and there exists $D\in\N$ such that, for all $d\geq D$, $X_\mu^{ss}(\Lb_{\lambda,\nu}+d\Lb_{(1),(|\mu|)})=\emptyset$. Thus, for all $d\geq D$,
\[ a_{\lambda+(d),\mu}^{\nu+(d|\mu|)}=\dim\left(\h^0(X_\mu^{ss}(\Lb_{\lambda,\nu}+d\Lb_{\alpha,\gamma}),\Lb_{\lambda,\nu}+d\Lb_{\alpha,\gamma})^G\right)=0. \]
We recover the result from Proposition \ref{prop_colmenarejo}.

\section{Application to the case of the hyperoctahedral group}\label{section_Bn}

\subsection{Notations and coefficients studied in this case}

For $n\geq 2$, we consider the group $W_n=(\Z/2\Z)^n\rtimes\gs_n$, which is the Weyl group in type $\mathrm{B}_n$ (if we see the root system of type $\mathrm{B}_n$ in $\R^n$ with basis $(\varepsilon_1,\dots,\varepsilon_n)$, $\gs_n$ acts by permuting the $\varepsilon_i$, whereas $1_i=(0,\dots,0,1,0,\dots,0)\in(\Z/2\Z)^n$ acts just by $\varepsilon_i\mapsto -\varepsilon_i$). It is called the hyperoctahedral group, and it is known (cf. \cite{wilson} or \cite{geck-pfeiffer}) that its rational irreducible complex representations can be built up from the ones of $\gs_n$ and are classified by double partitions of $n$. These are ordered pairs of partitions $(\alpha^+,\alpha^-)$ such that $|\alpha^+|+|\alpha^-|=n$.

\vspace{5mm}

When $(\alpha^+,\alpha^-)$ is a double partition, we choose to denote by $M_{\alpha^\pm}$ the associated irreducible representation of $W_{|\alpha^\pm|}$ (where $|\alpha^\pm|$ stands for $|\alpha^+|+|\alpha^-|$). Given two double partitions $(\alpha^+,\alpha^-)$ and $(\beta^+,\beta^-)$ of the same integer, consider the non-negative integers $c_{\alpha^\pm,\beta^\pm}^{\gamma^\pm}$ such that
\[ M_{\alpha^\pm}\otimes M_{\beta^\pm}=\bigoplus_{(\gamma^+,\gamma^-)}M_{\gamma^\pm}^{\oplus c_{\alpha^\pm,\beta^\pm}^{\gamma^\pm}}, \]
where the direct sum runs over all double partitions of $|\alpha^\pm|$.

\subsection{Schur-Weyl duality for $W_n$}

Let $V^+$ and $V^-$ be two complex vector spaces and set $V=V^+\oplus V^-$. Then the groups $\gl(V^\pm)=\gl(V^+)\times\gl(V^-)$ and $W_n$ act on $V^{\otimes n}$. (For $W_n$, $\gs_n$ acts simply by permuting the factors in $V^{\otimes n}$, and $1_i\in(\Z/2\Z)^n$ acts by multiplying by $-1$ the $i$-th factor in $V^{\otimes n}$.) Furthermore, these two actions commute and thus $\gl(V^\pm)\times W_n$ acts on $V^{\otimes n}$.

\begin{prop}
As a representation of $\gl(V^\pm)\times W_n$, $V^{\otimes n}$ decomposes as a direct sum of irreducible representations in the following way:
\[ V^{\otimes n}=\bigoplus_{(\alpha^+,\alpha^-)}V_{\alpha^\pm}(\gl(V^\pm))\otimes M_{\alpha^\pm}, \]
where the direct sum runs over all double partitions of $n$ such that $\ell(\alpha^+)\leq\dim(V^+)$ and $\ell(\alpha^-)\leq\dim(V^-)$. Moreover, $V_{\alpha^\pm}(\gl(V^\pm))$ denotes the irreducible representation $\Sc^{\alpha^+}(V^+)\otimes\Sc^{\alpha^-}(V^-)$ of $\gl(V^\pm)$.
\end{prop}

\begin{proof}
This result comes from \cite{sakamoto-shoji}.
\end{proof}

\vspace{5mm}

We now consider complex vector spaces $V_1=V_1^+\oplus V_1^-$ and $V_2=V_2^+\oplus V_2^-$ and we set $\gl(V_1^\pm)=\gl(V_1^+)\times\gl(V_1^-)$, $\gl(V_2^\pm)=\gl(V_2^+)\times\gl(V_2^-)$. Then, on the one hand,
\vspace{-5mm}
\begin{changemargin}{-3mm}{-3mm}
\[ \begin{array}{rcl}
V_1^{\otimes n}\otimes V_2^{\otimes n} & = & \left(\dps\bigoplus_{(\alpha^+,\alpha^-)}V_{\alpha^\pm}(\gl(V_1^\pm))\otimes M_{\alpha^\pm}\right)\otimes\left(\dps\bigoplus_{(\beta^+,\beta^-)}V_{\beta^\pm}(\gl(V_2^\pm))\otimes M_{\beta^\pm}\right)\\
 & = & \dps\bigoplus_{\alpha^\pm,\beta^\pm,\gamma^\pm}\left(\Sc^{\alpha^+}(V_1^+)\otimes\Sc^{\alpha^-}(V_1^-)\otimes\Sc^{\beta^+}(V_2^+)\otimes\Sc^{\beta^-}(V_2^-)\otimes M_{\gamma^\pm}\right)^{\oplus c_{\alpha^\pm,\beta^\pm}^{\gamma^\pm}},
\end{array} \]
\end{changemargin}
with the direct sum concerning all triples $(\alpha^+,\alpha^-),(\beta^+,\beta^-),(\gamma^+,\gamma^-)$ of double partitions of $n$.\\
On the other hand,
\[ V_1^{\otimes n}\otimes V_2^{\otimes n}=(V_1\otimes V_2)^{\otimes n}=\dps\bigoplus_{(\gamma^+,\gamma^-)}V_{\gamma^\pm}(\gl(V^\pm))\otimes M_{\gamma^\pm}, \]
where $\gl(V^\pm)=\gl(V^+)\times\gl(V^-)$ and $V^+=V_1^+\otimes V_2^+\oplus V_1^-\otimes V_2^-$, $V^-=V_1^+\otimes V_2^-\oplus V_1^-\otimes V_2^+$ (then $V_1\otimes V_2=V^+\oplus V^-$).\\
Moreover, one has a branching
\[ \underset{\text{denoted by }G}{\underbrace{\gl(V_1^+)\times\gl(V_1^-)\times\gl(V_2^+)\times\gl(V_2^-)}}\longrightarrow\underset{\text{denoted by }\hat{G}}{\underbrace{\gl(V^+)\times\gl(V^-)}} \]
and then a decomposition
\[ V_{\gamma^\pm}(\gl(V^\pm))=\bigoplus_{(\alpha^+,\alpha^-),(\beta^+,\beta^-)}\left(\Sc^{\alpha^+}(V_1^+)\otimes\Sc^{\alpha^-}(V_1^-)\otimes\Sc^{\beta^+}(V_2^+)\otimes\Sc^{\beta^-}(V_2^-)\right)^{\oplus\dots}. \]
Thus, by identification:

\begin{prop}
The coefficients $c_{\alpha^\pm,\beta^\pm}^{\gamma^\pm}$ are also the coefficients in the branching situation $G\rightarrow\hat{G}$, i.e. for all double partition $(\gamma^+,\gamma^-)$,
\vspace{-5mm}
\begin{changemargin}{-4mm}{-4mm}
\[ \Sc^{\gamma^+}(V^+)\otimes\Sc^{\gamma^-}(V^-)=\bigoplus_{(\alpha^+,\alpha^-),(\beta^+,\beta^-)}\left(\Sc^{\alpha^+}(V_1^+)\otimes\Sc^{\alpha^-}(V_1^-)\otimes\Sc^{\beta^+}(V_2^+)\otimes\Sc^{\beta^-}(V_2^-)\right)^{\oplus c_{\alpha^\pm,\beta^\pm}^{\gamma^\pm}}. \]
\end{changemargin}
\end{prop}

\vspace{5mm}

This new expression yields, by Schur's Lemma,
\vspace{-5mm}
\begin{changemargin}{-3mm}{-3mm}
\[ c_{\alpha^\pm,\beta^\pm}^{\gamma^\pm}=\dim\left(\Sc^{\gamma^+}(V^+)\otimes\Sc^{\gamma^-}(V^-)\otimes\Sc^{\alpha^+}(V_1^+)^*\otimes\Sc^{\alpha^-}(V_1^-)^*\otimes\Sc^{\beta^+}(V_2^+)^*\otimes\Sc^{\beta^-}(V_2^-)^*\right)^G. \]
\end{changemargin}
And finally we prove the following proposition:

\begin{prop}\label{link_line_bundles_Bn}
Let $(\alpha^+,\alpha^-)$, $(\beta^+,\beta^-)$, and $(\gamma^+,\gamma^-)$ be three double partitions of the same integer. Then there exist a complex reductive group $G$ acting on a projective variety $X$, and a $G$-linearised line bundle $\Lb_{\alpha^\pm,\beta^\pm,\gamma^\pm}$ over $X$ such that
\[ c_{\alpha^\pm,\beta^\pm}^{\gamma^\pm}=\dim\left(\h^0(X,\Lb_{\alpha^\pm,\beta^\pm,\gamma^\pm})^G\right). \]
\end{prop}

\begin{proof}
According to what precedes and thanks to Borel-Weil's Theorem, it is sufficient to consider complex vector spaces $V_1^+$, $V_1^-$, $V_2^+$, and $V_2^-$ such that $\dim(V_1^+)\geq\ell(\alpha^+)$, $\dim(V_1^-)\geq\ell(\alpha^-)$, $\dim(V_2^+)\geq\ell(\beta^+)$, and $\dim(V_2^-)\geq\ell(\beta^-)$. Then one sets
\vspace{-5mm}
\begin{changemargin}{-8mm}{-8mm}
\[ X=\fl(V_1^+)\times\fl(V_1^-)\times\fl(V_2^+)\times\fl(V_2^-)\times\fl(\underset{V^+}{\underbrace{V_1^+\otimes V_2^+\oplus V_1^-\otimes V_2^-}})\times\fl(\underset{V^-}{\underbrace{V_1^+\otimes V_2^-\oplus V_1^-\otimes V_2^+}}), \]
\end{changemargin}
\[ G=\gl(V_1^+)\times\gl(V_1^-)\times\gl(V_2^+)\times\gl(V_2^-), \]
and
\[ \Lb_{\alpha^\pm,\beta^\pm,\gamma^\pm}=\Lb_{\alpha^+}\otimes\Lb_{\alpha^-}\otimes\Lb_{\beta^+}\otimes\Lb_{\beta^-}\otimes\Lb^*_{\gamma^+}\otimes \Lb^*_{\gamma^-}. \]
\end{proof}

\subsection{Stability results and analogue of Murnaghan's stability}

\subsubsection{General result and examples}

According to the previous section, we find ourselves in the same situation as for Kronecker coefficients. As a consequence, the same proof as in Section \ref{general_section} can be applied here.

\begin{theo}\label{general_result_Bn}
If $\alpha^\pm=(\alpha^+,\alpha^-)$, $\beta^\pm=(\beta^+,\beta^-)$, and $\gamma^\pm=(\gamma^+,\gamma^-)$ are three double partitions such that
\[ \forall d\in\N^*, \: c_{d\alpha^\pm,d\beta^\pm}^{d\gamma^\pm}=1, \]
then the triple they form is stable in the sense that, for every double partitions $\lambda^\pm=(\lambda^+,\lambda^-)$, $\mu^\pm=(\mu^+,\mu^-)$, and $\nu^\pm=(\nu^+,\nu^-)$, the sequence
\[ \left(c_{\lambda^\pm+d\alpha^\pm,\mu^\pm+d\beta^\pm}^{\nu^\pm+d\gamma^\pm}\right)_{d\in\N} \]
stabilises for $d$ large enough.
\end{theo}

\paragraph{Example 1:} There is in this situation an analogue of Murnaghan's stability. It has already been observed and proven in \cite{wilson}, and we retrieve it here: according for instance to Proposition \ref{link_line_bundles_Bn}, we notice that
\[ c_{\big((1),\emptyset\big),\big((1),\emptyset\big)}^{\big((1),\emptyset\big)}=g_{(1),(1),(1)} \]
($\emptyset$ here stands for the empty partition, of size and length zero). Then we can apply the previous theorem to conclude that, for all double partitions $(\lambda^+,\lambda^-)$, $(\mu^+,\mu^-)$, and $(\nu^+,\nu^-)$ of the same integer, if we increase repetitively by one the first part of the partitions $\lambda^+$, $\mu^+$, and $\nu^+$, the associated sequence of coefficients $c$ eventually stabilises.

\paragraph{Example 2:} Let us consider the following triple of double partitions:
\[ \Big(\big((2),(2)\big),\big((2),(2)\big),\big((2),(2)\big)\Big) \]

\begin{lemma}
For all $d\in\N^*$,
\[ c_{d\big((2),(2)\big),d\big((2),(2)\big)}^{d\big((2),(2)\big)}=1. \]
\end{lemma}

\begin{proof}
Let us set $d\in\N^*$. We proved that the coefficient $c_{d\big((2),(2)\big),d\big((2),(2)\big)}^{d\big((2),(2)\big)}$ is the multiplicity of
\[ \sym^{2d}(V_1^+)\otimes\sym^{2d}(V_2^+)\otimes\sym^{2d}(V_1^-)\otimes\sym^{2d}(V_2^-) \]
in
\[ \sym^{2d}(V_1^+\otimes V_2^+\oplus V_1^-\otimes V_2^-)\otimes\sym^{2d}(V_1^+\otimes V_2^-\oplus V_1^-\otimes V_2^+) \]
(if $V_1^+$, $V_1^-$, $V_2^+$, and $V_2^-$ are large enough vector spaces). But we have (cf. for example \cite{fulton-harris}, Exercise 6.11)
\[ \begin{array}{rcl}
\sym^{2d}(V_1^+\otimes V_2^+\oplus V_1^-\otimes V_2^-) & = & \dps\bigoplus_{m+n=d}\sym^m(V_1^+\otimes V_2^+)\otimes\sym^n(V_1^-\otimes V_2^-)\\
 & = & \dps\bigoplus_{m+n=d, \; \lambda^+\vdash m, \; \lambda^-\vdash n}\Sc^{\lambda^+}(V_1^+)\otimes\Sc^{\lambda^+}(V_2^+)\\
 & & \otimes\Sc^{\lambda^-}(V_1^-)\otimes\Sc^{\lambda^-}(V_2^-).
\end{array} \]
And the same kind of formula exists for $\sym^{2d}(V_1^+\otimes V_2^-\oplus V_1^-\otimes V_2^+)$. Hence,
\[ \begin{array}{rl}
 & \sym^{2d}(V_1^+\otimes V_2^+\oplus V_1^-\otimes V_2^-)\otimes\sym^{2d}(V_1^+\otimes V_2^-\oplus V_1^-\otimes V_2^+)\\
 & \\
= & \dps\bigoplus_{(\lambda^+,\lambda^-),(\mu^+,\mu^-)\text{ s.t. }|\lambda^\pm|=|\mu^\pm|=2d}\Sc^{\lambda^+}(V_1^+)\otimes\Sc^{\mu^+}(V_1^+)\otimes\Sc^{\lambda^+}(V_2^+)\otimes\Sc^{\mu^-}(V_2^+)\\
 & \otimes\Sc^{\lambda^-}(V_1^-)\otimes\Sc^{\mu^-}(V_1^-)\otimes\Sc^{\lambda^-}(V_2^-)\otimes\Sc^{\mu^+}(V_2^-)\\
 & \\
= & \dps\bigoplus_{\lambda^+,\lambda^-,\mu^+,\mu^-,\nu_1,\nu_2,\nu_3,\nu_4}\big(\Sc^{\nu_1}(V_1^+)\otimes\Sc^{\nu_2}(V_2^+)\otimes\Sc^{\nu_3}(V_1^-)\otimes\Sc^{\nu_4}(V_2^-)\big)^{\oplus m^{\nu_1,\nu_2,\nu_3,\nu_4}_{\lambda^\pm,\mu^\pm}},
\end{array} \]
where this last sum runs over partitions verifying $|\lambda^\pm|=|\mu^\pm|=2d$, and
\[ m^{\nu_1,\nu_2,\nu_3,\nu_4}_{\lambda^\pm,\mu^\pm}=(\mathrm{LR})_{\lambda^+,\mu^+}^{\nu_1}\times(\mathrm{LR})_{\lambda^+,\mu^-}^{\nu_2}\times(\mathrm{LR})_{\lambda^-,\mu^-}^{\nu_3}\times(\mathrm{LR})_{\lambda^-,\mu^+}^{\nu_4}, \]
$(\mathrm{LR})$ denoting the Littlewood-Richardson's coefficients. Henceforth, the multiplicity of $\sym^{2d}(V_1^+)\otimes\sym^{2d}(V_2^+)\otimes\sym^{2d}(V_1^-)\otimes\sym^{2d}(V_2^-)$ is
\[ \sum_{\lambda^+,\lambda^-,\mu^+,\mu^-}(\mathrm{LR})_{\lambda^+,\mu^+}^{(2d)}\times(\mathrm{LR})_{\lambda^+,\mu^-}^{(2d)}\times(\mathrm{LR})_{\lambda^-,\mu^-}^{(2d)}\times(\mathrm{LR})_{\lambda^-,\mu^+}^{(2d)}, \]
where we take the sum over partitions such that $|\lambda^+|+|\lambda^-|=|\mu^+|+|\mu^-|=|\lambda^+|+|\mu^+|=|\lambda^+|+|\mu^-|=|\lambda^-|+|\mu^-|=|\lambda^-|+|\mu^+|=2d$, i.e. $|\lambda^+|=|\lambda^-|=|\mu^+|=|\mu^-|=d$. Then
\[ c_{d\big((2),(2)\big),d\big((2),(2)\big)}^{d\big((2),(2)\big)}=\sum_{\lambda^+,\lambda^-,\mu^+,\mu^-\vdash d}(\mathrm{LR})_{\lambda^+,\mu^+}^{(2d)}\times(\mathrm{LR})_{\lambda^+,\mu^-}^{(2d)}\times(\mathrm{LR})_{\lambda^-,\mu^-}^{(2d)}\times(\mathrm{LR})_{\lambda^-,\mu^+}^{(2d)}. \]
Finally, Littlewood-Richardson's rule shows that $(\mathrm{LR})_{\lambda,\mu}^{(2d)}=0$ unless $\lambda=\mu=(d)$. And in that last case, the coefficient is 1. This concludes the proof of the lemma.
\end{proof}

\begin{prop}
For every triple $\big((\lambda^+,\lambda^-),(\mu^+,\mu^-),(\nu^+,\nu^-)\big)$ of double partitions, the sequence
\[ \left(c_{\lambda^\pm+d\big((2),(2)\big),\mu^\pm+d\big((2),(2)\big)}^{\nu^\pm+d\big((2),(2)\big)}\right)_{d\in\N} \]
stabilises for $d$ large enough.
\end{prop}

\begin{proof}
It is a direct consequence of the previous lemma and of Theorem \ref{general_result_Bn}.
\end{proof}

\subsubsection{An example of an explicit bound}

We can also here compute in some special and not too difficult cases a bound for the stabilisation of the sequence of coefficients. We do this in the case analogous to Murnaghan's stability (Example 1).
As before we set $\Lb=\Lb_{((1),\emptyset),((1),\emptyset),((1),\emptyset)}$ and $\Mb=\Lb_{\lambda^\pm,\mu^\pm,\nu^\pm}$. If we consider the usual projection (cf. Section \ref{steps_computation})
\begin{center}
\begin{tikzpicture}
\node (un) at (0,0) {$X$};
\node (deux) at (2,0) {$\overline{X}$};
\node (trois) at (0,1.5) {$\Lb$};
\node (quatre) at (2,1.5) {$\overline{\Lb}$};
\node (cinq) at (-0.5,0) {$\pi:$};
\draw[->] (un)--(deux);
\draw[->] (trois)--(un);
\draw[->] (quatre)--(deux);
\draw[->, dash pattern=on 1mm off 1mm] (trois)--(quatre);
\end{tikzpicture}
\end{center}
such that $\Lb$ is the pull-back of an ample line bundle $\overline{\Lb}$, we notice that $\overline{X}$ and $\overline{\Lb}$ are exactly the same as in Section \ref{example_murnaghan}. Then we know that it is sufficient to determine when $\pi^{-1}(\overline{x})\subset X^{us}(\Mb+d\Lb)$ (same notation as in \ref{example_murnaghan} for $\overline{x}$). As a consequence, if we consider for instance the one-parameter subgroup
\[ \tau_0=\left.\left.\big(1,-1,0,\dots,0\text{ }\right|\text{ }0,\dots,0\text{ }\right|-1,1,0,\dots,0\text{ }\left|\text{ }0,\dots,0\big)\right. \]
of $G$, we have as before $\mu^{\overline{\Lb}}(\overline{x},\tau_0)=2$. And
\vspace{-5mm}
\begin{changemargin}{-1mm}{-1mm}
\[ \begin{array}{rcl}
\max_{x\in\pi^{-1}(\overline{x})}(-\mu^\Mb(x,\tau_0)) & = & -\lambda^+_1+\lambda^+_2-\mu^+_1+\mu^+_2+2\left(\nu^+_2-\nu^+_{\ell(\lambda^+)\ell(\mu^+)+\ell(\lambda^-)\ell(\mu^-)}\right)\\
 & & +\dps\sum_{k=1}^{\ell(\lambda^+)+\ell(\mu^+)-4}\left(\nu^+_{k+2}-\nu^+_{\ell(\lambda^+)\ell(\mu^+)+\ell(\lambda^-)\ell(\mu^-)-k}\right)\\
 & & +\dps\sum_{k=1}^{\ell(\lambda^-)+\ell(\mu^-)}\left(\nu^-_k-\nu^-_{\ell(\lambda^+)\ell(\mu^-)+\ell(\lambda^-)\ell(\mu^+)-k+1}\right).
\end{array} \]
\end{changemargin}

\begin{theo}
Let $(\lambda^+,\lambda^-)$, $(\mu^+,\mu^-)$, and $(\nu^+,\nu^-)$ be double partitions of the same integer. We set $m=\ell(\lambda^+)\ell(\mu^+)+\ell(\lambda^-)\ell(\mu^-)$, $n=\ell(\lambda^+)\ell(\mu^-)+\ell(\lambda^-)\ell(\mu^+)$, and
\vspace{-5mm}
\begin{changemargin}{-1.2cm}{-1.2cm}
\[ \mbox{\small $D=\left\lceil\dps\frac{1}{2}\left(-\lambda^+_1+\lambda^+_2-\mu^+_1+\mu^+_2+2\left(\nu^+_2-\nu^+_m\right)+ \dps\sum_{k=1}^{\ell(\lambda^+)+\ell(\mu^+)-4}\left(\nu^+_{k+2}-\nu^+_{m-k}\right)+\dps\sum_{k=1}^{\ell(\lambda^-)+\ell(\mu^-)}\left(\nu^-_k-\nu^-_{n-k+1}\right)\right)\right\rceil$.} \]
\end{changemargin}
Then, for all $d\geq D$ ($d\in\N$),
\[ c_{(\lambda^++(d),\lambda^-),(\mu^++(d),\mu^-)}^{(\nu^++(d),\nu^-)}=c_{(\lambda^++(D),\lambda^-),(\mu^++(D),\mu^-)}^{(\nu^++(D),\nu^-)}. \]
\end{theo}

\bibliographystyle{plain}
\bibliography{biblio_stab_kron_coefs}
\end{document}